\documentclass[11pt,a4paper]{article}
\usepackage[utf8]{inputenc}
\usepackage{amssymb,amsthm}
\usepackage{amsmath}
\usepackage[]{algorithm}
\usepackage{algpseudocode}
\usepackage{graphicx}
\usepackage{subcaption}
\usepackage[margin=3cm]{geometry}
\usepackage{comment}
\usepackage{color}
\usepackage{url}
\usepackage{ulem} 
\usepackage{accents}
\usepackage{enumitem}
\usepackage{tikz-cd}
\usepackage{bbm}
\usepackage{multirow}
\usepackage{array}
\usepackage[thinlines]{easytable}

\newtheorem{theorem}{Theorem}[section]
\newtheorem{definition}[theorem]{Definition}
\newtheorem{lemma}[theorem]{Lemma}

\newtheorem{proposition}[theorem]{Proposition}

\newcommand{\sgn}{\operatorname{sgn}}

{\begin{proof}[Beweis]}
	{\end{proof}}
\newtheorem{example}[theorem]{Example}

\theoremstyle{remark}
\newtheorem{remark}[theorem]{Remark}

\title{
\textbf{BOUNDARY CROSSING PROBLEMS AND  FUNCTIONAL TRANSFORMATIONS FOR ORNSTEIN-UHLENBECK PROCESSES}\\
}

\author{Aria Ahari$^1$\footnote{aria.ahari@warwick.ac.uk}, Larbi Alili$^1$\footnote{l.alili@warwick.ac.uk}, Massimiliano Tamborrino$^1$\footnote{massimiliano.tamborrino@warwick.ac.uk}}
\date{\small $^1$ Department of Statistics, University of Warwick, Coventry, CV4 7AL, United Kingdom.}

\begin{document}
\maketitle
\begin{abstract}
We are interested in the law of the first passage time of an Ornstein-Uhlenbeck process to time-varying thresholds. We show that this problem is connected to the laws of the first passage time of the process to members of a two-parameter family of functional transformations of a time-varying boundary. For specific values of the parameters, these transformations appear in a realisation of a standard Ornstein-Uhlenbeck bridge. We provide three different proofs of this connection. The first one is based on a similar result for Brownian motion, the second uses a generalisation of the so-called Gauss-Markov processes and the third relies on the Lie group symmetry method. 
       We investigate the properties of these transformations and study the algebraic and analytical properties of an involution operator which is used in constructing them. We also show that these transformations map the space of solutions of Sturm-Liouville equations into the space of solutions of the associated nonlinear ordinary differential equations. Lastly, we interpret our results through the method of images and give new examples of curves with explicit first passage time densities.   
\end{abstract}
\section{Introduction}

Let $U:=(U_t)_{t \geq 0} $ be a one-dimensional Ornstein-Uhlenbeck (OU for short) process defined  on a filtered probability space $(\Omega, (\mathcal{F})_{t \geq 0}, \mathcal{F}, \mathbb{P})$ as the unique solution to the following stochastic differential equation (SDE)
\begin{equation} \label{OU general SDE}
    dU_t = -k U_t dt +  dB_t, \; \; \; U_0={0},
\end{equation}
where $(B_t)_{t \geq 0}$ is a standard Brownian motion (BM) starting at 0 and $k \in \mathbb{R}$ is a constant. The OU process is a Gauss-Markov process with transition density function given by

\begin{equation} \label{OU density}
    p_{t}(x,y) :=  \frac{\partial }{\partial y} \mathbb{P}(U_t \leq y | U_0 = x) =     \frac{e^{kt}}{\sqrt{r(t)}} \phi\left( \frac{y e^{kt}-x 
    }{ \sqrt{r(t)}} \right), \quad x,y \in \mathbb{R},
\end{equation}
where $\phi(z)= e^{-\frac{z^2}{2}}/\sqrt{2 \pi}$, $z \in \mathbb{R}$, is the probability density function of the standard normal distribution and 
\[
\begin{aligned}
    r(t) & = (e^{2kt}-1)/2k, \quad t \geq 0,\\
   s(t) &=  \ln{(2kt+1)}/2k, \quad t \leq \zeta^{(k)},
\end{aligned}
\]
where
$$
 \zeta^{(k)} = \begin{cases}
- \frac{1}{2k} & \text{if} \quad k < 0;\\     
+\infty &  \text{otherwise}.
\end{cases}$$
It is well known by the Dambis, Dubins-Schwarz theorem (see, e.g., Theorem V.1.6 in \cite{Revuz yor brownian motion}), that the OU process can be written in terms of a time changed BM $(W_t)_{t \geq 0}$ as

\begin{equation}\label{OU brownian motion relation}
     U_t =  e^{-kt}  W_{r(t)}, \quad t \geq 0.
\end{equation}
Let $f \in \mathcal{C}([0,\infty), \mathbb{R})$ be such that $f(0) \neq 0$, with $\mathcal{C}(I, K)$ denoting the space of continuous functions from $I$ into $K$ for some intervals $I$ and $K \subseteq \mathbb{R}$. We are interested in the first passage time (FPT) of the OU to $f$ given by
\[
 T^f_k= \inf \{ t >0;\   U_{t} = f(t)\},
\]
with $\inf\{ \emptyset \} = \infty$. The main goal of this paper is to derive, through different methods, an explicit analytical expression linking the distribution of $T_{k}^f$ to that of  $T_{k}^{S^{\alpha,\beta}_k f}$. Here, the two-parameter family of curves $\{ S^{\alpha,\beta}_kf ; \alpha \neq 0, \beta \in \mathbb{R} \}$ is defined by

\begin{equation} \label{S hat transformation}
    {S}^{\alpha,\beta}_k f(t) = \left(\frac{1+ \alpha \beta r(t)}{\alpha} \right) 
    \left(2k \frac{\alpha^2 r(t)}{1+ \alpha \beta r(t)} + 1 \right)^{1/2} \; e^{-kt} \; f \left(s\left( \frac{ \alpha^2 r(t)}{1 + \alpha \beta r(t)} \right)\right), \quad t <  \zeta_{k,\alpha, \beta},
    \end{equation}
where 
$$
 \zeta_{k,\alpha, \beta} = \begin{cases}
s\left(- \frac{1}{\alpha \beta}\right) & \text{if}\; \alpha \beta <0, \; k \geq 0;\\     
s\left(- \frac{1}{\alpha \beta + 2 k \alpha^2 }\right) &\text{if}\; 0<\frac{2k}{\alpha \beta + 2k \alpha^2}<1,\; k<0;\\   
  + \infty & \text{otherwise}.
\end{cases}$$
By doing so, we generalize the results obtained for a BM in \cite{Alili patie ode}, {which can be immediately recovered from ours by letting $k\to 0$, i.e. $T^f_0, S^{\alpha,\beta}_0$ and $\zeta_{0,\alpha,\beta}$. To simplify the notation and for consistency with \cite{Alili patie ode}, we drop the subscript $0$ when referring to the BM.} 

To the best of our knowledge, explicit results for the FPT of the OU to $f$ only exist for constants {\cite{Alili patie OU,Ricciardi Sacerdote Sato}} or hyperbolic type boundaries {\cite{Buonocore OU hyperbolic boundary,Daniels}}. Further results for the boundary crossing problem of Gauss-Markov processes to moving boundaries have been obtained in, e.g., {\cite{Buonocore, Donofrio Pirozzi 2019,Durbin,Pirozzi}.}
{We also refer to \cite{Bluman and cole, Sacerdote} for applications of Lie symmetries to FPT problems. Our main result, stated in Theorem \ref{theorem OU transformation}, allows to map those results for the law of the FPT of $T_{k}^f$ to that of  $T_{k}^{S^{\alpha,\beta}_k f} $ for the OU process}. 
Such problems are of great interest, as the OU process has been used in many applications to model objects such as interest rates in finance or the evolution of the neuronal membrane voltages in neuroscience, see e.g. \cite{Alili patie OU} and citations therein. In this paper, we focus on the OU without drift, as the results for the OU process with drift can be directly obtained from our results after some transformations, {as discussed in Remark \ref{OU with drift curves}}. 

The paper is organised as follows. In Subsection \ref{Section21}, we introduce some notations and provide the key results for the  $S^{\alpha,\beta}$ operator for the BM. The OU functional setting is presented in Subsection \ref{Section22}. In Subsection {\ref{section OU bridge}}, we recall the different constructions of OU bridges and their properties. {In particular, the process $({S}_k^{1,-1/r(T)}U_t, 0< t \leq T)$ has the same law as an OU bridge of length $T$ from 0 to 0, for some $T>0$.} 
 {Section \ref{section main results} is devoted to the statement of Theorem \ref{theorem OU transformation}, which contains the main result of the paper, and two examples of its application. In Section \ref{section properties of transformation}, we discuss the properties of the ${S}^{\alpha,\beta}_kf$ transformation with its connection to a certain nonlinear differential equation (Lemma \ref{lemma nonlinear ODE})}, while in Section \ref{section proofs of theorem 3.1}, we prove Theorem \ref{theorem OU transformation} in three different ways. In the first proof, we use the relationship between the FPT of an OU and that of a BM. In the second one, we use a generalisation of the Gauss-Markov processes introduced in Section 3.2 of \cite{Alili patie ode} and find an analogue version of that proof in our case. In the third one, we use the Lie algebra to find the symmetries of the Fokker-Planck equation {or the Kolmogorov forward differential equation}
{
\begin{equation} \label{OU fokker planck equation}
    \frac{\partial h}{\partial t} = \frac{1}{2} \frac{\partial^2 h}{\partial x^2} + kx \frac{\partial h}{\partial x} + k  h.
\end{equation}}
Then, we use these symmetries to construct the function $h^{\alpha,\beta}_k$ of equation \eqref{h hat alpha beta equation}, derive our transformation $S^{\alpha,\beta}_kf$ and relate the FPT distribution of $T_{k}^f$ to that of $ T_{k}^{S^{\alpha,\beta}f}$, in Section \ref{section lie proof}. In Section \ref{section asymptotic behaviour of FPT densities}, we discuss the asymptotic distribution of $T_{k}^{S^{\alpha, \beta}_kf}$ and the transience of  the transformed curves $S^{\alpha,\beta}_kf$. We provide the analogue of the Kolmogorov–Erdös–Petrovski transience test \cite{Erdos} in the OU case and show its connection to the asymptotic behaviour of the FPT. Lastly, in Section {\ref{section method of images}}, we use the method of images to obtain new classes of boundaries yielding explicit FPT distributions and use our ${S}^{\alpha,\beta}_kf$ transformation (\ref{S hat transformation}) to produce new examples.  A limitation of this method is that it only works for boundaries with certain properties given in Lemma \ref{lemma characterisation}.  

As we were finalising the paper, we discovered that 
the $S^{\alpha,\beta}_k$ transformation, a variant of the boundary crossing identity (\ref{FPT dist relation}) in Theorem \ref{theorem OU transformation} and the Lie symmetries (\ref{symmetry OU})  had previously appeared in \cite{Dmitry Lie} (our (\ref{S hat transformation}) can be obtained by setting $A=k^2$ and $B=0$ in equation (39) therein), using the Lie approach. When comparing our results, we found misprints in one of their Lie symmetries and boundary crossing identity, as discussed in Section \ref{section lie proof}.


\section{Notation and preliminaries}\label{Section2}
 {We first introduce some functional spaces, transformations and related results for the BM, as in \cite{Alili patie ode}, in Subsection \ref{Section21}, and then define the corresponding functional objects for the OU  in  Subsection \ref{Section22}. We end this section by providing different 
representations of OU bridges and highlighting their connection to our functional transformations.}
 
\subsection{Brownian motion setting}\label{Section21}
 We start by introducing a nonlinear operator $\tau$ defined on the space of functions whose reciprocals are square integrable in some (possibly infinite) interval
of $\mathbb{R}^+=[0,\infty)$ by 
$$\tau f(t) = \int_{0}^{t} f^{-2}(z) dz,$$
and use it to define 
$$A_{\infty}=\bigcup\limits_{a >0} \bigcup\limits_{b > 0} A(a,b) \; \; \text{and} \; \; A(a,b) = \left\{\pm f \in  \mathcal{C}([0,a], \mathbb{R}^+) : \;  \tau f\left(a\right) = b \right\},$$
where $a,b\in\mathbb{R}^+$. {In \cite{Alili patie ode}, the authors derived the following relationship between the laws of the FPTs of $T^f$ and $T^{S^{\alpha,\beta} f}$, 
\begin{equation}\label{BM FPT relation}
    \mathbb{P}(T^{S^{\alpha, \beta}f} \in dt) = \alpha^3 (1+ \alpha \beta t)^{- \frac{5}{2}} e^{- \frac{\alpha \beta}{2(1+ \alpha \beta t)} (S^{\alpha, \beta}f(t))^2} S^{\alpha, \beta}  (\mathbb{P}( T^f \in dt)), \quad t<\zeta_{0, \alpha, \beta},
\end{equation}
 where $S^{\alpha,\beta}$ is the two-parameter family of transformations $S^{\alpha,\beta}:A_{\infty}\to A_{\infty}$ given by
 \begin{equation} \label{tranformation brownian motion}
    S^{\alpha,\beta} f(t) = \left(\frac{1+ \alpha \beta t}{\alpha} \right) f \left(\frac{ \alpha^2 t}{1 +  \alpha \beta t} \right), \quad  \alpha\neq 0, \beta \in \mathbb{R}.
\end{equation}
Equation \eqref{BM FPT relation} extends a previous result by the same authors for a relationship between the laws of the FPT of $T^f$ and $T^{S^{1,\beta} f}$, with the one-parameter family of transformations $\{S^{1,\beta} f, \beta \in \mathbb{R}\}$ obtained using the construction of Brownian bridges, see \cite{Alili patie time inversion}. The connection to Brownian bridges naturally appears as $(S^{1,-1/T}(B)_t , 0 \leq t < T )$ is a Brownian bridge of length $T$ from 0 to 0, see page 64 of \cite{Borodin Salminen}.}
{Besides deriving \eqref{BM FPT relation} in \cite{Alili patie ode}, the authors showed also that the transformation $S^{\alpha,\beta}$ can be obtained as 
\begin{equation}\label{S0BM}
    S^{\alpha,\beta}= \Sigma \circ \Pi^{\alpha, -\beta} \circ \Sigma
\end{equation}
where $\circ$ denotes the composition operator. Here, $\Sigma: A_{\infty} \to A_{\infty}$ is the involution operator,  i.e.,  $\Sigma \circ \Sigma  = Id$, specified by 
\[
\Sigma f(t) = \frac{1}{f(\rho \circ \tau f(t))},
\]
where $\rho$ is the inversion operator acting on the space of continuous monotone functions i.e., $ \rho f \circ f(t) = t$. 
For  $\alpha \in \mathbb{R}^*=\mathbb{R}\backslash\{0\}, \beta \in \mathbb{R}$, $\Pi^{\alpha,\beta}: A_{\infty}\to A_{\infty}$ is the family of nonlinear operators  given by
\begin{equation} \label{Pi}
     \Pi^{\alpha, \beta}f(t)= f(t)(\alpha + \beta \tau f(t)).
\end{equation}
As explained in the beginning of Section 2 of \cite{Alili patie ode} and Appendix 8 of \cite{Revuz yor brownian motion}, the operators \eqref{Pi} are closely related to the Sturm-Liouville equation
\begin{equation} \label{heat SL equation}
        \phi'' = \mu \phi,
\end{equation}
where $\mu$ denotes a positive Radon measure on $\mathbb{R}^+$ and $\phi''$ is the second derivative in the sense of distributions. In fact, if $\phi$ solves (\ref{heat SL equation}), then the vectorial space $\{\Pi^{\alpha,\beta} \phi ; \alpha\neq 0, \beta \in \mathbb{R}\}$ is the set of other solutions to the same equation. Moreover, all positive solutions are convex and described by the set $\{\Pi^{\alpha,\beta} \varphi ; \alpha>0, \beta \geq 0\}$, where $\varphi$ is the unique, positive, decreasing solution such that $\varphi(0)=1$. 
\\
Moreover, it was noted in \cite{Alili patie ode} that $S^{\alpha, \beta} \circ S^{\alpha', \beta'} = S^{\alpha \alpha', \alpha \beta' + \frac{\beta}{\alpha'}}$, for all couples $(\alpha, \beta)$ and $(\alpha', \beta') \in \mathbb{R}^{*} \times \mathbb{R}$, and that $(S^{1, \beta})_{\beta  \geq 0}$ is a semi-group, while $(S^{1, \beta})_{\beta  \in \mathbb{R}}$ and $(S^{{e^\alpha}, 0})_{\alpha  \in \mathbb{R}}$ are groups.}

\subsection{Ornstein-Uhlenbeck setting}\label{Section22}\label{section notations}
We shall now define the operators of interest on the space of functions
\begin{equation} \label{funtional space infinite}
     A_{k,\infty}=\bigcup\limits_{a >0} \bigcup\limits_{b > 0} A_k(a,b), 
\end{equation}
where $a,b \in \mathbb{R}^+$, and
\begin{equation} \label{funcional space finite}
    {A}_k(a,b) = \left\{\pm f \in  \mathcal{C}([0,a], \mathbb{R}^+) : \;  \tau f\left(\sgn(k)s\big(\sgn(k)a\big)\right) = \frac{b}{1 - \mathbbm{1}_{\{k<0\}} 2kb} \right\},
\end{equation}
with $\sgn(k)$ being the sign of $k$. Note that $A_{k,\infty}$ is the set of continuous functions which are of constant sign on some nonempty interval $[0, l], \; l>0$.
{Let the isomorphic nonlinear operators $\Lambda_k$ and $\Sigma_k$ be defined, on $A_{k,\infty}$, by 
\begin{equation} \label{lamda_k operator}
    \Lambda_k f (t)= e^{ks(t)} f(s(t))
\end{equation}
and
\begin{equation} \label{sigma hat composition}
    {\Sigma_k} = {\Lambda_k}^{-1} \circ \Sigma \circ {\Lambda_k},
\end{equation}
respectively.
Note that the inverse of $\Lambda_k$ is $\Lambda_k^{-1}f(t)= e^{-kt} f(r(t))$}. {Here, $\Lambda_k$ maps the curves from the OU setting to the corresponding curves in the BM setting, with limits $\lim_{k \to 0}\Lambda_k = Id$, and $\lim_{k\to 0} \Sigma_k = \Sigma$.} From \eqref{sigma hat composition}, we can immediately see that
$\Sigma_k$ can also be represented as
\begin{equation} \label{sigma hat functional}
    {\Sigma}_kf(t) = \frac{e^{-k \rho \circ \tau f(r(t)) - kt}}{f(\rho \circ \tau f(r(t)))}.
\end{equation}
Inspired by \eqref{S0BM} in the BM setting, for $\alpha \in \mathbb{R}^*$ and $\beta \in \mathbb{R}$, define ${S}_k^{\alpha, \beta}:  A_{k,\infty}\to A_{k,\infty}$ by

\begin{equation} \label{S hat decomposition}
    S^{\alpha, \beta}_k  = \Sigma_k \circ \Pi^{\alpha, -\beta} \circ \Sigma_k.
\end{equation}
{Alternatively, $S^{\alpha, \beta}_k$ can be also obtained as $\Lambda_k^{-1}\circ S^{\alpha,\beta}\circ \Lambda_k$, as shown in the proof of Lemma \ref{lemma nonlinear ODE}. This result provides us with an intuitive interpretation of this family of transformations. First, we map}
{the boundary $f$ for the OU process to its corresponding boundary for the standard BM using the $\Lambda_{k}$ transformation. Then, we use $S^{\alpha,\beta}$ on this curve for the standard BM, and finally we revert it back to the OU problem via $\Lambda^{-1}_k$, as illustrated in the following figure.} 
\newline
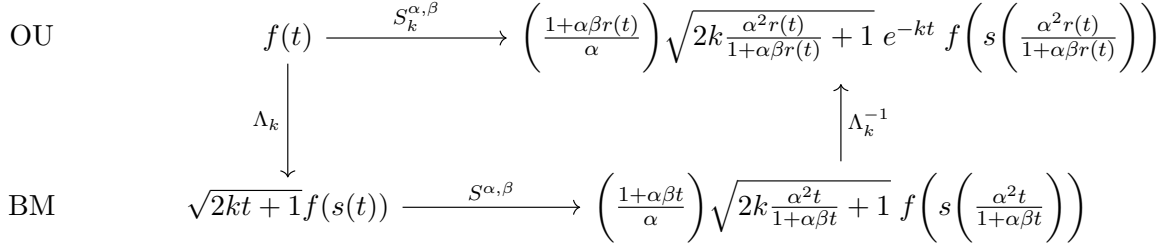
\begin{figure}[h]
    \centering
    \[
\begin{tikzcd}[ampersand replacement=\&, row sep=large, column sep = large]
\textrm{OU} \&   f(t)\arrow{r}{{S}^{\alpha, \beta}_k} \arrow[swap]{d}{{\Lambda_k}} \&   \bigg(\frac{1+ \alpha \beta r(t)}{\alpha} \bigg) 
    \sqrt{2k \frac{\alpha^2 r(t)}{1+ \alpha \beta r(t)} + 1} \; e^{-kt} \; f \bigg(s\bigg( \frac{ \alpha^2 r(t)}{1 + \alpha \beta r(t)} \bigg)\bigg)\\%
\textrm{BM} \&      \sqrt{2kt+1} f(s(t)) \arrow{r}{S^{\alpha, \beta}} \&  \bigg(\frac{1+ \alpha \beta t}{\alpha} \bigg) 
    \sqrt{2k \frac{\alpha^2 t}{1+ \alpha \beta t} + 1} \; f \bigg(s\bigg( \frac{ \alpha^2 t}{1 + \alpha \beta t} \bigg)\bigg)  \arrow[swap]{u}{{\Lambda}_{k}^{-1}}
\end{tikzcd}
\]  
    \caption{Flow chart of $S^{\alpha,\beta}_k$.}
    \label{flow chart of transformation}
\end{figure}  

\subsection{Ornstein-Uhlenbeck bridges} \label{section OU bridge}

As mentioned in the introduction, $S^{\alpha,\beta}_k$ for specific values of $\alpha$ and $\beta$ relates to a representation of the standard OU bridge. We shall now define this process and then recall its different representations, using the detailed analysis of both Wiener and OU bridges provided in \cite{Barczy OU bridge representations,Borodin Salminen}.
\begin{definition}
An OU bridge $U^{\text{br}} = \{U^{\text{br}}_{t}: t \in [0,T] \}$ from a to b, of length T, is characterised by the following properties:
\begin{enumerate}[label=(\roman*)]
    \item $U^{\text{br}}_{0}=a$ and $U^{\text{br}}_{T}=b$ (each with probability 1).
    \item $U^{\text{br}}$ is a Gaussian process.
    \item $\mathbb{E}[U^{\text{br}}_{t}] = a \frac{\sinh{(k(T-t))}}{\sinh{(kt)}} + b \frac{\sinh{(kt)}}{\sinh{(kT)}}$.
    \item $\text{Cov}(U^{\text{br}}_{s}, U^{\text{br}}_{t}) = \frac{\sinh{(ks)} \sinh{(k(T-t))}}{k \sinh{(kT)}},\quad 0 \leq s \leq t < T$.
    \item {The paths of $U^{\text{br}}$} are almost surely continuous.
\end{enumerate}
\end{definition}
In what follows, we present three different representations of OU bridges. First, consider the following linear SDE
\begin{equation} \label{OU bridge sde}
    dU^{\text{br}}_{t} =  \left(- k \coth{(k(T-t))} U^{\text{br}}_{t} + k \frac{b}{\sinh{(k(T-t))}} \right) dt + dB_t, \; \; \; 0 \leq t < T,
\end{equation}
with initial condition $U^{\text{br}}_{0} = a$. This has a unique strong solution given by 
\[
    U_{t}^{ir} := \begin{cases}
 a \frac{\sinh{(k(T-t))}}{\sinh{(kT)}} + b \frac{\sinh{(kt)}}{\sinh{(kT)}} + \int_{0}^{t} \frac{\sinh{(k(T-t))}}{\sinh{(k(T-s))}} dB_s & \text{if}\quad 0 \leq t < T,\\      
b & \text{if} \quad t = T.
\end{cases}
\]
This is referred to as the integral representation (ir) of the OU bridge. The following anticipative (av) and the space-time (st) versions can be obtained by using the different representations of the Wiener bridge,
\[
    U_{t}^{av} = a \frac{\sinh{(k(T-t))}}{\sinh{(kT)}} + b \frac{\sinh{(kt)}}{\sinh{(kT)}} +\left(U_t -   \frac{\sinh{(kt)}}{\sinh{(kT)}} U_T \right), \quad 0 \leq t < T,
\]
\[
    U_{t}^{st} = a \frac{\sinh{(k(T-t))}}{\sinh{(kT)}} + b \frac{\sinh{(kt)}}{\sinh{(kT)}} + e^{-kt} \frac{(r(T) - r(t))}{r(T)} W_{\frac{r(T) r(t)}{ r(T) - r(t)}}, \quad 0 \leq t < T.
\]
{Now by setting $a=b=0$, we see that $U_t^{st} = {S}^{1,-1/r(T)}_k U_t$}. One thing to notice is that the anticipative  and space-time versions are only weak solutions to (\ref{OU bridge sde}), {since the former requires information about the random variable $U_T$ and the latter is adapted to a different filtration, see Section 1 of \cite{Barczy OU bridge representations}.} Of course, $U^{br}$ can be thought of as $(U_t, t \leq T)$ conditioned on $U_T = b$. For conditioned processes and Markovian bridges, see  \cite{Pitman yor markovian bridges} and \cite{Revuz yor brownian motion}.

\section{Main result and examples} \label{section main results}

We are now ready to state the main result of this paper which relates
the distributions of the family of stopping times $(T_{k}^{S^{\alpha,\beta}_k f})_{\alpha \in \mathbb{R}^*,\beta \in \mathbb{R}}$ to that of $T_{k}^{f}$. {We assume that $f \in \mathcal{C}^1((0, \infty), \mathbb{R}^+)$, i.e. it belongs to the space of continuously differentiable functions from $(0,\infty)$ into $\mathbb{R}^+$. This assumption is linked to Strassen's result for the BM in \cite{strassen} and is crucial in ensuring that the distribution of  $T_{k}^f$ has a continuous density with respect to the Lebesgue measure on $(0, \infty)$.}


\begin{theorem} \label{theorem OU transformation}
        Let $f \in C^1([0,\infty), \mathbb{R})$ be such that $f(0)\neq0$. Let $\alpha\in\mathbb{R}^*, \beta \in \mathbb{R}$. Then, for $t < \zeta_{k,\alpha, \beta}$, we have the following relationship
        
        \begin{equation} \label{FPT dist relation}
    \begin{aligned}
       &  \mathbb{P}\Big(  T^{{S}^{\alpha,\beta}_k f}_k  \in dt  \Big) = \\
       &  e^{3kt} \alpha^3 (1+ \alpha \beta r(t))^{-5/2} \bigg( 2k \frac{\alpha^2 r(t)}{1+ \alpha \beta r(t)} + 1 \bigg)^{-3/2}
        e^{- \frac{\alpha \beta }{2(1 + \alpha \beta r(t))} (S^{\alpha,\beta}_k f)^2 e^{2kt}} S^{\alpha,\beta}_k\Big(  \mathbb{P}( T^{ f}_k \in dt  )\Big), \\
    \end{aligned}
    \end{equation}
    { where $S^{\alpha,\beta}_k$ is given by \eqref{S hat transformation}}.
\end{theorem}

\begin{remark} \label{OU with drift curves}
{As mentioned in the introduction, focusing on the driftless OU is not restrictive. Consider the process $U^{\mu} := (U_t^{\mu})_{t \geq 0}$, defined for each fixed $t>0$ by 
\[
U_t^{\mu} : = U_t + \mu(1 - e^{-kt}),
\]
is called an OU process with drift $\mu$, see Remark 2.5 of \cite{Alili patie OU}. For $f:[0,\infty) \to  \mathbb{R}$ define $f^*(t) := f(t) - \mu(1 - e^{-kt})$, $t \geq 0$. Then, the FPT of $U^{\mu}$ hitting the curve $f$ equals to the FPT of $U$ hitting the curve $f^*$.
}
\end{remark}

\begin{example}
Theorem 3.1 in \cite{Alili patie OU} gives an expression for the FPT density of an OU process starting at $0$ hitting a constant threshold $f(t)=a>0$. Writing $p_k^a(t)dt = \mathbb{P}\Big(  T_{k}^{a} \in dt  \Big)$, we have
 \[
 		p_k^a(t) = -k e^{-ka^2/2} \sum_{j=1}^{\infty} \frac{D_{\nu_{j,-a \sqrt{t}}}(0)}{D'_{\nu_{j, -a \sqrt{t}}} (-a \sqrt{2k})} \exp(-k \nu_{j, -a \sqrt{2k} }t),
 \]
 where $D_\nu (.)$ is the parabolic cylinder function with index $\nu \in \mathbb{R}$,  
$\nu_{j,b}$ the ordered sequence of positive zeros of $D_{\nu}(b)$, $D_{\nu_{j,-a \sqrt{t}}}(0) = 2^{\frac{\nu_{j,-a \sqrt{t}}}{2}} \frac{\Gamma (1/2)}{\Gamma\left((1 - \nu_{j,-a \sqrt{t}})/2\right)}$ and $D'_{\nu_{j,b}}(b) = \frac{\partial D_{\nu} (b)}{\partial \nu}|_{\nu = \nu_{j,b}}$. Here, $D_\nu (z) = 2^{-\nu/2} e^{-z^2/4} H_{\nu}(z/\sqrt{2})$, where $H_{\nu}(.)$ is the Hermit function, as found on page 285 of \cite{Lebedev}.
Now, we use $S^{\alpha,\beta}_k$ to get the following family of curves
\[
\begin{aligned}
    S^{\alpha,\beta}_k a(t) & = a \left(\frac{1+ \alpha \beta r(t)}{\alpha} \right) 
    \left(2k \frac{\alpha^2 r(t)}{1+ \alpha \beta r(t)} + 1 \right)^{1/2} \; e^{-kt} \\
    & = \frac{a e^{-kt}}{2k \alpha} \left( (\alpha \beta e^{2kt} - \alpha \beta + 2k)\left[(2k\alpha^2+\alpha \beta) e^{2kt} - 2k \alpha^2 -\alpha \beta +2k\right] \right)^{1/2}.
\end{aligned}
\]
Such curves for $\beta=1, \; k=1/2, \; a=1$ and different values of $\alpha$ are plotted in Figure \ref{figure transformed barrier 1}. By our Theorem \ref{theorem OU transformation}, we can write
\begin{eqnarray}
\nonumber p_k^{S^{\alpha,\beta}_k a}(t) & =&   \frac{-k e^{2kt} \alpha^2(1+ \alpha \beta r(t))^{-5/2}}{\big(2\alpha^2r(t)+ \alpha\beta r(t) + 1\big)} \exp\left(- \frac{\alpha \beta }{2(1 + \alpha \beta r(t))} (S^{\alpha,\beta}_k a(t))^2 \; e^{2kt} - ka^2/2\right)\\
\label{neweq} && \times  \sum_{j=1}^{\infty} \frac{D_{\nu_{j,-a \sqrt{t}}}(0)}{D'_{\nu_{j, -a \sqrt{t}}} (-a \sqrt{2k})} \exp\left(-k \nu_{j, -a \sqrt{2k} } \; s\left( \frac{ \alpha^2 r(t)}{1 + \alpha \beta r(t)} \right) \right).
\end{eqnarray}
In the special case of $\alpha \beta = 2k,$ we get the family of curves 

\[S^{\alpha,\beta}_k f(t)   =  \frac{a}{\alpha} \sqrt{ (\alpha^2 +1) e^{2kt} - \alpha^2},\] 
with FPT density given by
\[
	p_k^{S^{\alpha,\beta}_k f}(t) = \frac{-k \alpha^2 e^{-k a^2 \left(\frac{\alpha^2 +1}{\alpha^2} \right) e^{2kt}+ \frac{k a^2}{2}+kt}}{(\alpha^2 +1)e^{2kt}-\alpha^2}  \sum_{j=1}^{\infty} \frac{D_{\nu_{j,-a \sqrt{t}}}(0)}{D'_{\nu_{j, -a \sqrt{t}}} (-a \sqrt{2k})} e^{-k \nu_{j, -a \sqrt{2k} } \; s\left( \frac{ \alpha^2 r(t)}{1 + \alpha \beta r(t)} \right)}.
\]
{An alternative expression for \eqref{neweq} can be obtained by using Theorem \ref{theorem OU transformation} with the semi-explicit expression for $p_k^a$ provided in \cite{Ricciardi Sacerdote Sato}.}

 \begin{figure}	 \includegraphics[width=.6\textwidth]{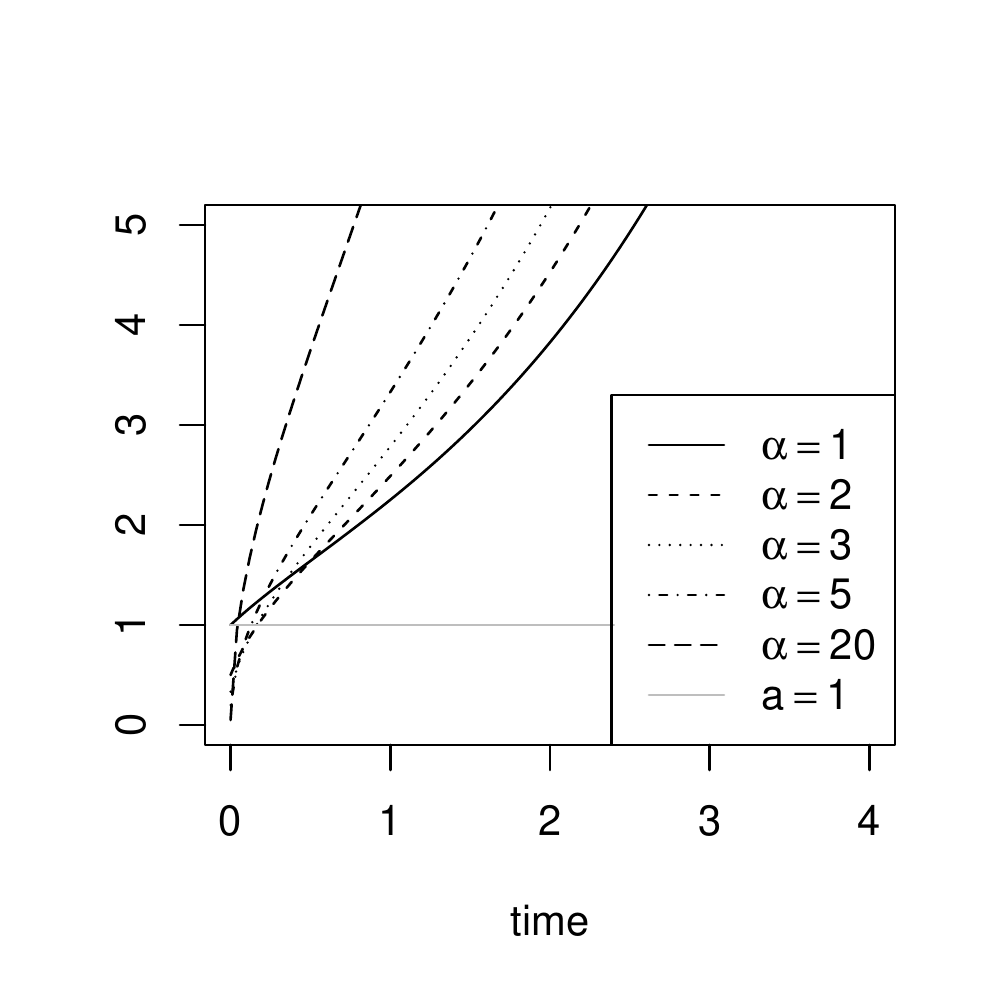}
	\centering
	\caption{${S}^{\alpha,1}_{1/2} a$ for $a=1, \; \beta=1, \; k=1/2,$ and different values of $\alpha$, in particular $\alpha$=1\;(continuous line),\;  2\;(short-dashed line),\; 3\;(dotted line),\; 5\;(dashed-dotted line),\; 20\;(long-dashed line).}
	\label{figure transformed barrier 1}
\end{figure}
\end{example}

\begin{example}
    {By Theorem 2.1 in \cite{Groeneboom}, the law of the first time when a BM hits a parabolic curve of the form $f_{a,b}(t)= a + bt^2$, for $a,b>0$, is given by 
   \[
    \mathbb{P}(T^{f_{a,b}} \in dt) = e^{-2/3 \; b^2 t^3 } h_{b,a}(t)dt, \quad t>0,
    \]
    where $h_{b,a}: \mathbb{R}^{+} \to \mathbb{R}^{+}$ is characterised by the  Laplace transform
    \[
    \int_{0}^{\infty} e^{-\lambda u} h_{b,a}(u)du = \frac{Ai\left((4b)^{1/3}a + (2b^2)^{-1/3} \lambda\right)}{Ai\left((2b^2)^{-1/3} \lambda\right)},
    \quad \lambda>0, 
    \]
    and $Ai$ denotes the Airy function of the first kind, as introduced on page 136 of \cite{Lebedev}. By the scaling property of the BM, one has
\[
T^{f_{a,b}} = a^2 T^{f_{1,ba^3}},
\]    
and so we can use the series expansion given in example 4.1.2 of \cite{Alili patie time inversion} for the density of $T^{f_{1,ba^3}}$ to get
\begin{equation} \label{FPT of BM to parabola}
     \mathbb{P}(T^{f_{a,b}} \in dt) = 2(c ba^3)^2 e^{- \frac{2}{3} b^2 t^3} \sum_{j=0}^{\infty} \frac{Ai(z_j + 2c ba^3)}{Ai'(z_j)} e^{-z_j t/a^2} \frac{dt}{a^2},
\end{equation}
where $(z_j)_{j \geq 0}$ is the decreasing sequence of negative zeros of the function $Ai(.)$ and $c=(2 ba^3)^{-1/3}$. Now, the FPT problem for this curve is equivalent to the FPT problem of the OU hitting the curve $g(t) := \Lambda_{k}^{-1}f_{a,b}(t) = e^{-kt} (a + b r^2(t))$, with its FPT law given by 
    \[
    \mathbb{P}(T_{k}^g \in dt) = 2(c ba^3)^2 e^{-2/3 b^2 r(t)^3} \sum_{j=0}^{\infty} \frac{Ai(z_j + 2c ba^3)}{Ai'(z_j)} e^{-z_j r(t)/a^2}  e^{2kt} \frac{dt}{a^2},
    \]
    which is obtained using (\ref{FPT of BM to parabola}) and (\ref{OU brownian stopping time relation}).
   Applying the transformation $S^{\alpha,\beta}_k$ to $g$, we get the family of curves
   \[
 S^{\alpha,\beta}_k g(t) =  \alpha^{-1}(1 + \alpha \beta r(t))  e^{-kt} \left[a + b  \left( \frac{ \alpha^2 r(t)}{1 + \alpha \beta r(t)} \right)^2 \right]  .
    \]
 Recalling the notation $p_k^g(t)dt = \mathbb{P}\Big(  T_{k}^{g} \in dt  \Big)$, applying Theorem \ref{theorem OU transformation} and simplifying,  we obtain
\begin{eqnarray*}
p_k^{S^{\alpha,\beta}_k g}(t) &  = &2e^{2kt} \alpha^2 (ba^2 c)^2(1+ \alpha \beta r(t))^{-3/2}  e^{- \frac{\alpha \beta }{2(1 + \alpha \beta r(t))} (S_{k}^{\alpha,\beta}g(t))^2 \; e^{2kt}}  \; 
 e^{-2/3 \; b^2 \left( \frac{ \alpha^2 r(t)}{1 + \alpha \beta r(t)} \right)^3} \\
      	  & &\times \sum_{k=0}^{\infty} \frac{Ai(z_k + 2c ba^3)}{Ai'(z_k)} e^{- \frac{z_k}{a^2}  \frac{ \alpha^2 r(t)}{1 + \alpha \beta r(t)} }.
\end{eqnarray*}
    }
\end{example}

\section{Properties of transformations} \label{section properties of transformation}

We now state a result that gives an insight into our $S^{\alpha,\beta}_k$ transformation. 

\begin{lemma} \label{lemma nonlinear ODE}
\begin{enumerate}[label=(\arabic*)]
    \item  Let $(\alpha, \beta)$ and $(\alpha', \beta') \in \mathbb{R}^{*} \times \mathbb{R}$ and $A_{k,\infty}$ be defined as in \eqref{funtional space infinite}. The mapping $S^{\alpha,\beta}_k: A_{k,\infty} \to A_{k,\infty}$ defined by (\ref{S hat decomposition}) admits representation \eqref{S hat transformation} and satisfies 
    
    \begin{equation} \label{S hat composition property}
            S^{\alpha, \beta}_k \circ S^{\alpha', \beta'}_{k} = S^{\alpha \alpha', \alpha \beta' + \frac{\beta}{\alpha'}}_{k}. 
    \end{equation}

     In particular, $(S^{1,\beta}_{k})_{\beta\in\mathbb{R}}$  {and $( S^{e^\alpha,0}_{k})_{\alpha\in\mathbb{R}}$ are groups}.
    Furthermore, for $a,b>0$ and $A_{k}(a,b)$ defined as in \eqref{funcional space finite}, if $f \in A_{k}(a,b)$, then $S^{\alpha, \beta}_{k}f \in A_{k}(a_{\alpha, \beta}, b_{\alpha, \beta}^f)$, where we set
\[
 a_{\alpha, \beta} = \begin{cases}
 \frac{a}{\alpha^2 - \alpha \beta a} & \text{if} \quad \alpha^2 - \alpha \beta a>0, \; k \geq 0 ;\\   
  \frac{a}{\alpha^2 - (\alpha \beta + 2k \alpha^2)a}  & \text{if} \quad 0<\frac{2k a}{(\alpha \beta + 2k \alpha^2)a - \alpha^2}<1, \; k<0 ;\\   
  + \infty & \text{otherwise},
\end{cases}
\]
and we wrote for $f \in A_{k,\infty}$
\[
 b_{\alpha, \beta}^f = \begin{cases}
b  & \text{if} \quad\alpha^2 - \alpha  \beta b >0, \; k\geq0; \\  
\frac{b}{1-2kb} & \text{if} \quad  0<\frac{2k b}{(\alpha \beta + 2k \alpha^2)b - \alpha^2}<1, \; k<0; \\
 \rho \circ \tau f\left(s\left(\frac{\alpha}{\beta}\right)\right) & \text{otherwise}.
\end{cases}
\]
Observe that $a_{\alpha, \beta} \to r(\zeta_{k,\alpha, \beta})$ as $a \to \infty$.

    \item Let $\mu$ be a positive Radon measure on $\mathbb{R}^{+}$. Then, there exists a unique positive and differentiable function $\mathfrak{f}$ with $\mathfrak{f}(0)=1$ and $\mathfrak{f}'+ k \mathfrak{f} + \frac{k}{(2k \tau \mathfrak{f}+1)\mathfrak{f}}> 0$, which satisfies the following nonlinear differential equation

    \begin{equation} \label{nonlinear ODE}
    f^3 f'' - k^2 f^4= \frac{-\mu(s(\tau f)) + k^2}{(2 k (\tau f)+1)^2},
    \end{equation}
{where $f''$ is the second derivative in the sense of distributions.} Moreover,   $\{S^{\alpha,\beta}_k \mathfrak{f}; \alpha>0, \beta \in \mathbb{R} \}$ spans the set of positive solutions of (\ref{nonlinear ODE}). 
\end{enumerate}
\end{lemma}

Before proving {Lemma} \ref{lemma nonlinear ODE}, we start by discussing some of the properties of $\Sigma_{k}$ and $\Pi^{\alpha,\beta} $. 

\begin{proposition} \label{prop sigma hat properties}
   \begin{enumerate}[label=(\arabic*)]
   For any $a,b>0$ and $f \in A_{k,\infty}$, we have the following assertions.
       \item $\Sigma_{k}$ is an involution operator, that is $\Sigma_{k} \circ \Sigma_{k} = Id$ on $A_{k,\infty}$.
       \item $\tau f (t) = r(\rho \circ \tau \circ \Sigma_{k} f(r(t)))$.
       \item $\Sigma_{k}(A_{k}(a,b)) = A_{k}(b, a)$.
       \item $\Pi^{\alpha,\beta} \Lambda_{k}  = \Lambda_{k} \Pi^{\alpha,\beta} $.
   \end{enumerate}
\end{proposition}

\begin{proof}
 \begin{enumerate}[label=(\arabic*)]
       \item  {The involution property follows from  (\ref{sigma hat composition}) and the fact that $\Sigma$ is an involution.}  
       \item {By using \eqref{sigma hat functional}, we obtain }
       $$\tau \circ \Sigma_{k}f(t) = \int_{0}^{t} f^2(\rho \circ \tau f(r(y))) e^{2ky + 2k \rho \circ \tau f(r(y))} dy = \int_{0}^{\rho \circ \tau f(r(t))} e^{2kz} dz = r(\rho \circ \tau f(r(t))),$$ 
       {where we used a change of variables in the second equality. The assertion follows using the fact that $\Sigma_{k}$ is an involution.}
       \item First, consider $k<0$ and take $f \in A_{k}(a,b)$. Observe that $\rho \circ \tau f(b/(1-2kb)) = -s\big(-a\big)$. Now, by part (2),  $\tau \circ \Sigma_{k}f\left(-s\left(-b\right)\right) = r(\rho \circ \tau f(b/(1-2kb)) = a/(1-2ka)$ and so $\Sigma_{k}f \in A_{k}(b,a)$. Similarly, the result holds for $k>0$.
       
       \item $\tau \Lambda_{k} f(t) = \int_{0}^{t} \left[e^{2k s(y)} f^2(s(y))\right]^{-1} dy = \tau f(s(t))$, which follows by a change of variable. The identity can easily be derived now.
   \end{enumerate}
\end{proof}
\begin{proof}[Proof of Lemma \ref{lemma nonlinear ODE}]
\begin{enumerate}[label=(\arabic*)]
    \item Starting from the definition of $S_k^{\alpha,\beta}$ in \eqref{S hat decomposition} and using the decomposition of $\Sigma_k$ in \eqref{sigma hat composition}, identity (4) from Proposition \ref{prop sigma hat properties} and the decomposition of $S^{\alpha,\beta}$ given in \eqref{S0BM}, we can write
\[
   S_{k}^{\alpha, \beta}  =   \Sigma_{k} \circ \Pi^{\alpha, -\beta} \circ \Sigma_{k} = \Lambda_{k}^{-1} \circ  (\Sigma  \circ \Pi^{\alpha, -\beta} \circ \Sigma) \circ \Lambda_{k}  =  \Lambda_{k}^{-1} \circ S^{\alpha,\beta} \circ \Lambda_{k} .
\]
    
Now, if we use the functional definitions of the operators, we get the representation \eqref{S hat transformation}, as illustrated in the previously reported Figure \ref{flow chart of transformation}. {Next, property (\ref{S hat composition property}) follows from the similar group property of the family of operators $(S^{\alpha, \beta})_{(\alpha , \beta) \in \mathbb{R}^* \times \mathbb{R}}$, as stated in Section \ref{Section21}, see also Proposition 2.3 of \cite{Alili patie ode}.} Lastly, using the identity 
$\tau S_{k}^{\alpha,\beta}f(t) = \tau f \left(s\left( \frac{ \alpha^2 r(t)}{1 + \alpha \beta r(t)} \right) \right)$ which follows readily by a change of variables, we get that if $f \in A_{k}(a,b)$, then $S_{k}^{\alpha, \beta}f \in A_{k}(a_{\alpha, \beta}, b_{\alpha, \beta}^f)$.
 
 \item As discussed in Section  \ref{section notations}, if $\phi$ solves the Sturm-Liouville equation (\ref{heat SL equation}), then the vectorial space $\{\Pi^{\alpha,\beta} \phi ; \alpha, \beta \in \mathbb{R}\}$ is the set of solutions to the same equation. Moreover, all positive solutions are convex and {are} described by the set $\{\Pi^{\alpha,\beta} \varphi ; \alpha>0, \beta \geq 0\}$, where $\varphi$ is the unique, positive, decreasing solution such that $\varphi(0)=1$. Now, we need to show that the image of (\ref{nonlinear ODE}) by $\Sigma_{k}$ is (\ref{heat SL equation}) and vice-versa. Let $\phi(t) = \Sigma_{k}f(t)$. Then, we obtain $\phi(s(\tau f(t))) f(t) = e^{-k s(\tau f(t)) - kt}$. Differentiating both side once, we get
\[
\phi'(s(\tau f(t))) = - e^{-k s(\tau f(t))-kt} (2k\tau f(t)+1) \left(kf(t) +  f'(t) + \frac{k}{(2k\tau f(t)+1)f(t)}  \right).
\]
Differentiating again, we obtain 
\[
 \phi''(s(\tau f(t)))  = 
 - (2k\tau f(t)+1)^2 e^{-k s(\tau f(t)) - kt}  \left( f^2(t) f''(t)  - k^2 f^3(t) - \frac{k^2}{(2k\tau f(t)+1)^2 f(t)}\right),
\]
So if $\phi$ satisfies (\ref{heat SL equation}), we get (\ref{nonlinear ODE}) and vice versa. Now, by letting $\mathfrak{f} = \Sigma_{k} \varphi$, we see that it satisfies the required properties. 
\end{enumerate}
\end{proof}


\section{Proofs of Theorem \ref{theorem OU transformation}} \label{section proofs of theorem 3.1}

We provide three different proofs of Theorem \ref{theorem OU transformation}. The first method is a direct approach that uses results from \cite{Alili patie ode}. The second approach uses a generalisation of the so-called Gauss-Markov processes and gives an insight into some of the results given in {Lemma} \ref{lemma nonlinear ODE}. The third proof relies on the Lie group techniques applied to the {Fokker-Planck equation \eqref{OU fokker planck equation}.} {We also}
 mention that, similarly as in the BM case, it is sufficient to only consider the case where $\alpha>0$. This is because, by the symmetry of the BM and the OU process starting at 0 without drift, hitting for the first time negative or positive valued thresholds has the same probability. Hence, for any $(\alpha,\beta) \in \mathbb{R}^{*} \times \mathbb{R}$, we have $
T_{k}^{S_{k}^{\alpha,\beta}f} \stackrel{\text{d}}{=} T_{k}^{S_{k}^{|\alpha|,\sgn(\alpha)\beta}f},$
where $ \stackrel{\text{d}}{=}$ denotes equality in distribution.

\subsection{First proof of Theorem \ref{theorem OU transformation}}\label{Section5.1}

\begin{proof}[\unskip\nopunct]
 We aim to use \eqref{BM FPT relation} which connects the FPT distribution of $  S^{\alpha,\beta} f$ to that of $f$. We can connect the FPTs of OU and BM in the following manner,

\begin{equation} \label{OU brownian stopping time relation}
	 T_{k}^{f}  = \inf \{t>0 ; \;  U_t = f(t) \}  
		 =  \inf \{s(t)> 0 ; \; W_{t} = e^{ks(t)} f(s(t)) \} 
		 = s(T^{\Lambda_{k} f}).
\end{equation}
Note that $T_{k}^{S_{k}^{\alpha,\beta} f} = s(T^{S^{\alpha,\beta} \Lambda_{k} f}).$ Writing $p_k^{f}(t)  dt = \mathbb{P}\left( T_{k}^{ f} \in dt  \right)$ and $p^{f}(t)  dt = \mathbb{P}\left( T^{f} \in dt \right)$, we get
\begin{equation} \label{f hat to g alpha}
   p_k^{ S_{k}^{\alpha,\beta} f}(t) = e^{2kt} \; p^{S^{\alpha,\beta} \Lambda_{k} f}(r(t)).
\end{equation}
Now, using (\ref{BM FPT relation}) and (\ref{tranformation brownian motion}), we can rewrite the right hand side as follows:
\begin{equation} \label{g alpha beta to g}
\begin{aligned}
 p^{S^{\alpha,\beta} \Lambda_{k} f}(r(t)) & = \alpha^3 (1+ \alpha \beta r(t))^{-5/2} e^{- \frac{\alpha \beta }{2(1 + \alpha \beta r(t))} (S^{\alpha,\beta} \Lambda_{k} f(r(t)))^2} S^{\alpha,\beta} \Big( p^{\Lambda_{k} f}(r(t))\Big) \\
& = \alpha^2 (1+ \alpha \beta r(t))^{-3/2} e^{- \frac{\alpha \beta }{2(1 + \alpha \beta r(t))} (S^{\alpha,\beta} \Lambda_{k} f(r(t)))^2}    p^{\Lambda_{k} f} \left(\frac{ \alpha^2 r(t)}{1 + \alpha \beta r(t)} \right).
\end{aligned}
\end{equation}
By using (\ref{OU brownian stopping time relation}), we get
\begin{equation} \label{T^g to T^f}
  p^{\Lambda_{k} f} \left(\frac{ \alpha^2 r(t)}{1 + \alpha \beta r(t)} \right) = \frac{1}{2k \frac{\alpha^2 r(t)}{1+ \alpha \beta r(t)} + 1} \; p_k^{f} \left( s\left( \frac{ \alpha^2 r(t)}{1 + \alpha \beta r(t)} \right) \right).
\end{equation}
Notice that $S^{\alpha,\beta} \Lambda_{k} f(r(t)) = e^{kt} S_{k}^{\alpha,\beta}f(t) $. Now, plugging (\ref{T^g to T^f}) into (\ref{g alpha beta to g}) and then into (\ref{f hat to g alpha}), we get
\[
\begin{aligned}
  p_k^{ S_{k}^{\alpha,\beta}f}(t) &
 = e^{2kt} \alpha^2 \frac{(1+ \alpha \beta r(t))^{-3/2}}{\big(2k \frac{\alpha^2 r(t)}{1+ \alpha \beta r(t)} + 1 \big)}  e^{- \frac{\alpha \beta }{2(1 + \alpha \beta r(t))} (S_{k}^{\alpha,\beta}f(t))^2 \; e^{2kt}}  \; p_k^{f} \left( s\left( \frac{ \alpha^2 r(t)}{1 + \alpha \beta r(t)} \right) \right)\\
& = e^{3kt} \alpha^3 (1+ \alpha \beta r(t))^{-5/2} \bigg(2k \frac{\alpha^2 r(t)}{1+ \alpha \beta r(t)} + 1 \bigg)^{-3/2} e^{- \frac{\alpha \beta }{2(1 + \alpha \beta r(t))} (S_{k}^{\alpha,\beta}f(t))^2 \; e^{2kt}} \; S_{k}^{\alpha,\beta} \bigg(p_k^{f}(t) \bigg),
\end{aligned}
\]
where we used (\ref{S hat transformation}) to obtain the last equation.
\end{proof}

\subsection{Second proof of Theorem \ref{theorem OU transformation}}
 We take $\phi \in A_{k}(a,b) \cap AC([0,b))$, where $AC([0,b))$ is the space of absolutely continuous functions on $[0,b)$. We introduce the process $X:=(X_t)_{0 \leq t <b}$ to be the generalised Gauss-Markov process of OU type with parameters $(\phi,k)$, which is defined as the unique strong solution to the following SDE
 \[
 dX_t = \left(\frac{\phi'(t)}{\phi(t)} + k \right )X_t dt + e^{kt} dB_t, \quad 0 \leq t < b,
 \]
with $X_0=x \in \mathbb{R}$, i.e.,
\[
X_t = \phi(t) e^{kt} \left( x +  \int_{0}^{t} \frac{1}{\phi(s)} dB_s \right), \quad 0 \leq t < b.
\]
We also denote by $\mathbb{P}^{(\phi,k)} = \left(\mathbb{P}^{(\phi,k)}_{x} \right)_{x \in \mathbb{R}}$ the family of probability measures corresponding to the process X. We assume throughout that $\phi(0) = 1$. Notice that when $\phi(t) = e^{-kt}$, we get $X_t = W_{r(t)}$, where $(W_t, t \geq 0)$ is another BM.

\begin{lemma} \label{lemma stopping time relations}
For any $y \in \mathbb{R}$, set $T_y = \inf \{0<t<b ; \;  \phi(t) e^{kt} \   \int_{0}^{t} \frac{1}{\phi(s)} dB_s = y  \}$. Then, for any $f \in  A_k(a,b)$, setting $\phi = \Sigma_{k}f$, the identity 
\[
    T_{k}^{f} = s(\tau \phi(T_{1}))
\]
holds almost surely.
\end{lemma}

\begin{proof}
Using (\ref{OU brownian motion relation}), we have a.s.

\[
\begin{aligned}
T_{1} & = \inf \{t >0 ; \;    \phi(t) e^{kt} W_{\tau \phi(t)} = 1 \} 
= \inf \left\{t>0 ; \; U_{s(\tau \phi(t))} =  \frac{e^{-kt - k s(\tau \phi(t))}}{\phi(t)} \right\} = \rho \circ \tau \phi (r(T_{k}^{f}))
\end{aligned} 
\]
and the result follows.
\end{proof}

Next, we introduce the following notation
\[
    H_{t}^{k}(x) = \sqrt{\frac{\alpha \phi(t)}{\Pi^{\alpha, \beta} \phi(t)} } 
    \exp\left(\frac{\beta}{2} \frac{x^2 e^{-2kt}}{\phi(t) \Pi^{\alpha, \beta} \phi(t)}\right).
\]
Our aim now is to show that the parametric families of distributions $\left(\mathbb{P}^{(\Pi^{\alpha, \beta} \phi,k)}\right)_{(\alpha,\beta) \in \mathbb{R}^* \times \mathbb{R}}$ of generalised
Gauss-Markov processes are related by some space-time harmonic transformations. The proof of the following proposition is similar to that of Lemma 3.2 in \cite{Alili patie ode}. However, the proof therein has typos, so we give the full proof here.

\begin{proposition} \label{Prop gauss absolute cont measure}
For $(\alpha,\beta) \in \mathbb{R}^* \times \mathbb{R}$ and $\phi$ as above, the process $(H_t^{k}(X_t))_{0 \leq t < a_{\alpha, -\beta}^{\phi}}$ is a $\mathbb{P}^{(\phi,k)}$-martingale. Furthermore, the absolute-continuity relationship
\[
    d \mathbb{P}^{(\Pi^{\alpha, \beta} \phi,k)}_{x | \mathcal{F}_t} = \frac{H_t^{k}(X_t)}{H_0^{k}(x)}   d \mathbb{P}^{( \phi,k)}_{x | \mathcal{F}_t}
\]
holds for all $x \in \mathbb{R}$ and $t < a_{\alpha, -\beta}^{\phi}$. Consequently, for any reals x and y, we have
\[
    \mathbb{P}^{(\Pi^{\alpha, \beta} \phi,k)}_{x} (T_y \in dt) = \frac{H_t^k(y)}{H_0^k(x)}   \mathbb{P}_{x}^{(\phi,k)}(T_y \in dt), \quad t < a_{\alpha, -\beta}^{\phi}.
\]
\end{proposition}

\begin{proof}
 From the It\^{o} formula, we have
\[
\frac{\beta}{2} \frac{X_t^2 e^{-2kt}}{\phi(t) \Pi^{\alpha, \beta} \phi(t)} = \beta \int_{0}^{t} \frac{X_s e^{-ks}}{\phi(s)\Pi^{\alpha, \beta} \phi(s)} dB_s - \frac{\beta^2}{2} \int_{0}^{t} \left(\frac{X_s e^{-ks}}{\phi(s)\Pi^{\alpha, \beta} \phi(s)}\right)^2 ds + \frac{1}{2} \ln \left(\frac{\alpha+ \beta \tau \phi(t)}{\alpha}\right).
\]
Now, as $\mathbb{E} [e^{-\frac{\lambda}{2} B_t^2}] = (1 +\lambda t)^{-1/2}, \; \lambda > -1/t $, we deduce that, for all $t<a_{\alpha,-\beta}^{\phi}$, we have
\[
\mathbb{E}[H_{t}^{k}(X_t)] = \mathbb{E}\left[\sqrt{\frac{\alpha}{\alpha + \beta \tau \phi}} e^{-\frac{\beta}{2} \frac{W_{\tau \phi}^2}{\alpha + \beta \tau \phi}}\right] = 1
\] 
Hence, it is a true martingale. Next, notice that
\[
d \left< \beta \int_{0}^{.} \frac{X_s e^{-ks}}{\phi(s)\Pi^{\alpha, \beta} \phi(s)} dB_s, X_.   \right>_{t} = \beta \frac{X_t }{\phi(t)\Pi^{\alpha, \beta} \phi(t)} dt,
\]
where $\left<.,.\right>$ is the quadratic variation process. Also,
\[
\frac{(\Pi^{\alpha, \beta} \phi(t))'}{\Pi^{\alpha, \beta} \phi(t)} = \frac{\phi'(t)}{\phi(t)} + \frac{\beta}{\phi(t)\Pi^{\alpha, \beta} \phi(t)},
\]
and hence the absolute continuity relationship follows by an application of Girsanov’s theorem. Now, Doob's optional stopping theorem implies that

\[
\begin{aligned}
       \mathbb{P}^{(\Pi^{\alpha, \beta} \phi,k)}_{x} (T_y \leq t) & = \mathbb{E}_{x}^{(\phi,k)} \left[\mathbbm{1}_{\{T_y \leq t\}} \frac{H_t^{k}(X_t)}{H_0^{k}(x)}\right] 
         = \mathbb{E}_{x}^{(\phi,k)} \left[\mathbbm{1}_{\{T_y \leq t\}} \mathbb{E}_{x}^{(\phi,k)} \left[ \frac{H_t^{k}(X_t)}{H_0^{k}(x)} \bigg| \mathcal{F}_{t \wedge T_y}\right]\right] \\
       & = \mathbb{E}_{x}^{(\phi,k)} \left[\mathbbm{1}_{\{T_y \leq t\}} \frac{H_{T_y}^{k}(y)}{H_0^{k}(x)}\right],
\end{aligned}
\]
where $\mathbbm{1}_A$ is the indicator function of the set $A$. The result follows by differentiation. 
\end{proof}
Now, we are ready to give the second proof of Theorem \ref{theorem OU transformation}.

\begin{proof}[Second proof of Theorem \ref{theorem OU transformation}]
Let $\phi = \Sigma_{k}f$, $f_k^{\alpha,\beta} = S_{k}^{\alpha,\beta}f$ and thus, by definition, $f_k^{\alpha,\beta} = \Sigma_{k} \circ \Pi^{\alpha, -\beta} \phi$. Since $\Sigma_{k}$ is an involution, from Lemma \ref{lemma stopping time relations} we get $T_{k}^{f_k^{\alpha,\beta}} = s(\tau \circ \Pi^{\alpha, -\beta} \phi(T_1))$ a.s. Using identity (2) from Proposition \ref{prop sigma hat properties}, we get 
\[
    r(\rho \circ \tau \circ \Pi^{\alpha, -\beta} \phi(r(t))) = r(\rho \circ \tau \circ \Sigma_{k}f_k^{\alpha,\beta} (r(t)) = \tau f_k^{\alpha,\beta}(t).
\]
Now, for $t < \zeta_{k, \alpha, \beta}$, writing $p_k^{f_k^{\alpha,\beta}}(t)  dt = \mathbb{P}\Big( T_{k}^{  f_k^{\alpha,\beta}} \in dt  \Big)$ and $\Tilde{p}^{(\phi , k)}(t)  dt =  \mathbb{P}^{(\phi,k)}_{0} \left(T_1 \in dt \right)$, we get
\[
\begin{aligned}
p_k^{f_k^{\alpha,\beta}}(t)  & = \frac{1}{(2k (\tau f_k^{\alpha,\beta} )+1)(f_k^{\alpha,\beta})^2}  \Tilde{p}^{(\Pi^{\alpha, -\beta} \phi , k)} \left( s(\tau f_k^{\alpha,\beta} (t)) \right) \\
&  = \frac{1}{(2k (\tau f_k^{\alpha,\beta} )+1)(f_k^{\alpha,\beta})^2}   \Tilde{p}^{(\Pi^{\alpha, -\beta} \phi , k)} \left( s \left(\tau f \left(s\left( \frac{ \alpha^2 r(t)}{1 + \alpha \beta r(t)} \right) \right) \right) \right),
\end{aligned}
\]
where we used the identity $\tau f_k^{\alpha,\beta}(t) = \tau f \left(s\left( \frac{ \alpha^2 r(t)}{1 + \alpha \beta r(t)} \right) \right)$ which follows readily by a change of variable.
Now, using Proposition \ref{Prop gauss absolute cont measure} with the identities
\[
\begin{aligned}
  \left(\phi \left( s \circ \tau f \left(s\left( \frac{ \alpha^2 r(t)}{1 + \alpha \beta r(t)} \right)  \right) \right) \right)^{-1} & = \left(\phi \left(\rho \circ \tau \phi \left(r \circ  s\left( \frac{ \alpha^2 r(t)}{1 + \alpha \beta r(t)} \right)  \right) \right)\right)^{-1}\\
      & = e^{k \rho \circ \tau \Sigma_{k} f \left( \frac{ \alpha^2 r(t)}{1 + \alpha \beta r(t)} \right) + k s \left( \frac{ \alpha^2 r(t)}{1 + \alpha \beta r(t)} \right)}  f \left(s \left( \frac{ \alpha^2 r(t)}{1 + \alpha \beta r(t)} \right) \right)\\
      & =  e^{k s (\tau f_k^{\alpha, \beta}(t)) + k s \left( \frac{ \alpha^2 r(t)}{1 + \alpha \beta r(t)} \right)}  f \left(s \left( \frac{ \alpha^2 r(t)}{1 + \alpha \beta r(t)} \right) \right) ,
\end{aligned}
\]
and
\[
\begin{aligned}
     (\Pi^{\alpha, - \beta} \phi( s (\tau f_k^{\alpha, \beta}(t))))^{-1} & = (\Pi^{\alpha, - \beta} \phi( \rho \circ \tau \circ \Pi^{\alpha, -\beta} \phi (r(t)) ))^{-1} = e^{k \rho \circ \tau  \Pi^{\alpha, -\beta} \phi (r(t)) + kt} f_k^{\alpha, \beta} (t)\\
     & = e^{k s (\tau f_k^{\alpha, \beta}(t)) + kt} f_k^{\alpha, \beta} (t),
\end{aligned}
\]
yields, for any $t < \zeta_{k, \alpha, \beta}$,
\[
\begin{aligned}
 & \Tilde{p}^{(\Pi^{\alpha, -\beta} \phi , k)} \left(  s \left(\tau f \left(s\left( \frac{ \alpha^2 r(t)}{1 + \alpha \beta r(t)} \right) \right) \right) \right)  = \\
& \sqrt{1+ \alpha \beta r(t)}   e^{- \frac{\alpha \beta }{2(1 + \alpha \beta r(t))} (f_k^{\alpha,\beta}(t))^2 \; e^{2kt}} \Tilde{p}^{(\phi , k)} \left( s \left(\tau f \left(s\left( \frac{ \alpha^2 r(t)}{1 + \alpha \beta r(t)} \right) \right) \right) \right).
\end{aligned}
\]
Hence,
\[
p_k^{f_k^{\alpha,\beta}}(t) = \frac{e^{- \frac{\alpha \beta }{2(1 + \alpha \beta r(t))} (f_k^{\alpha,\beta}(t))^2 \; e^{2kt}}\sqrt{1+ \alpha \beta r(t)}}{(2k (\tau f_k^{\alpha,\beta} )+1)(f_k^{\alpha,\beta})^2}  \Tilde{p}^{(\phi , k)} \left(  s \left(\tau f \left(s\left( \frac{ \alpha^2 r(t)}{1 + \alpha \beta r(t)} \right) \right) \right) \right).
\]
Using Lemma \ref{lemma stopping time relations} again, we finally get (\ref{FPT dist relation}).
\end{proof}

\subsection{Third proof of Theorem \ref{theorem OU transformation} via the Lie group symmetries} \label{section lie proof}

We now provide the last proof for Theorem \ref{theorem OU transformation} using the Lie group symmetries approach. For Lie Group theory, we refer to {\cite{ Bluman book,Peter olver lie}}. In general, this technique can be used to find solutions to differential equations with new boundary conditions from known ones. For example, the Lie point symmetries of the heat equation
\[
     \frac{\partial h}{\partial t} = \frac{1}{2} \frac{\partial^2 h}{\partial x^2}
\]
can be found in Section 3.3 of \cite{Alili patie ode}. Before going through the proof, we first discuss the connections between the below boundary value problems (\ref{bvps}) corresponding to the heat equation and the OU Fokker-Planck equation. 
Set
\[
D^f = \{(x,t) \in \mathbb{R \times \mathbb{R}^+} : x \leq f(t)\}.
\] 
We introduce the following boundary value problem, 
\begin{equation}\label{bvps}
    \mathcal{H}_k(f) : = 
\begin{cases}
  \frac{\partial}{\partial t}h(x,t)=  k \frac{\partial}{\partial x}(x h(x,t)) +\frac{1}{2}\frac{\partial^2}{\partial x^2}h(x,t) \quad \text{on}\; D^f;\\      
  h(f(t),t)=0 \quad \text{for all} \; t>0; \\
  h(.,0) = \delta_0(.) \quad \text{on} \; (-\infty,f(0)).
  \end{cases}
\end{equation}
Note that the first equation in $\mathcal{H}_k(f)$ is the OU Fokker-Planck equation \eqref{OU fokker planck equation}. As $k \to 0$, we recover the heat equation and obtain the corresponding boundary value problem for the BM $\mathcal{H}(f)$. By Proposition 5.4.3.1 of \cite{Yor blanc method for finance 2009}, solutions to $\mathcal{H}(f)$ and $\mathcal{H}_k(f)$ admit the following probabilistic representations:
\begin{equation} \label{OU BM probability solution representation}
      h(x,t)dx = \mathbb{P}(W_t \in dx,\; t<T^f)  \; \; \; \text{and} \; \; \; h_k(x,t)dx = \mathbb{P}(U_t \in dx,\; t<T^{f}_k).
\end{equation}
Using (\ref{OU brownian motion relation}) and (\ref{OU brownian stopping time relation}), we connect the two solutions directly in the following way:
\[
\mathbb{P}(W_t < x, \; t< T^{\Lambda_{k} f})  = \mathbb{P}(e^{-ks(t)} W_{t} < e^{-ks(t)} x , \; s(t) < s(T^{\Lambda_{k} f})) = \mathbb{P}(U_{s(t)} < e^{-ks(t)}x, \; s(t) < T^{f}_k)
\]
and so
\begin{equation} \label{OU BM solution relation}
\begin{aligned}
\; \; & h(x,t) = e^{-ks(t)} h_k(e^{-ks(t)}x,s(t)) 
    \; \; \; \text{and} \; \; \;
  h_k(x,t) & = e^{kt} h(e^{kt}x,r(t)),
\end{aligned}
\end{equation}
where $h$ now denotes a solution to $\mathcal{H}(\Lambda_{k} f)$. This shows that the solutions to the Fokker-Plank equation of the OU on $D^f$ are directly connected to the solutions to the heat equation on $D^{\Lambda_{k} f}$. As will become clear in the third proof of Theorem \ref{theorem OU transformation} at the end of this section, the aim is to find a solution to the Fokker-Plank equation of the OU such that it vanishes on our desired transformed boundary $S_{k}^{\alpha,\beta}f$. To do this, we first present a proposition for the OU which resembles Proposition 3.5 in \cite{Alili patie ode} for the Brownian motion. 

\begin{proposition} \label{proposition h hat alpha beta expression}
   Let $h_k$ be the solution to the boundary value problem $\mathcal{H}_k(f)$ in \eqref{bvps}. Then, for any $\alpha>0$, $\beta \in \mathbb{R}$, the following mapping
  \begin{equation} \label{h hat alpha beta equation}
      h^{\alpha,\beta}_k(x,t) = \frac{\alpha e^{kt}}{\sqrt{1+\alpha \beta r(t)}} e^{- \frac{\alpha \beta e^{2kt} x^2}{2(1+ \alpha \beta r(t))} - k s\left(\frac{\alpha^2 r(t)}{1+\alpha \beta r(t)}\right)} h_{k}\left(\frac{\alpha e^{-k s\left(\frac{\alpha^2 r(t)}{1+\alpha \beta r(t)}\right)+kt}}{1+\alpha \beta r(t)} x, s\left(\frac{\alpha^2 r(t)}{1+\alpha \beta r(t)} \right)  \right),
  \end{equation}
  for $t < \zeta_{k, \alpha, \beta}$, is the solution to the boundary value problem $\mathcal{H}_k(S_{k}^{\alpha,\beta}f)$. 
\end{proposition}
\begin{proof}
Firstly, if $f$ is an infinitely continuously differentiable function, then so is its transformation $S_{k}^{\alpha,\beta}f$.  From Proposition 3.5 in \cite{Alili patie ode}, the following function
\[
    h^{\alpha, \beta} (x,t) = \frac{\alpha}{\sqrt{1+ \alpha \beta t}} e^{- \frac{\alpha \beta x^2}{2(1+ \alpha \beta t)}} h\left(\frac{\alpha x}{1 + \alpha \beta t}, \frac{\alpha^2 t}{1 + \alpha \beta t} \right),
\]
is a solution to $\mathcal{H}(S^{\alpha,\beta} \Lambda_{k} f)$ whenever $h$ solves $\mathcal{H}(\Lambda_{k} f)$. Now, using relation (\ref{OU BM solution relation}), we define $\Tilde{h}^{\alpha,\beta}$ as $\Tilde{h}^{\alpha,\beta}= e^{kt} h^{\alpha,\beta}(e^{kt}x, r(t)) $. Then, this expression is a solution to the boundary value problem $\mathcal{H}_k(S_{k}^{\alpha,\beta}f)$. Using relation (\ref{OU BM solution relation}) again, we can write
\[
h\left(\frac{\alpha e^{kt} x}{1+\alpha \beta r(t)}, \frac{\alpha^2 r(t)}{1+\alpha \beta r(t)}\right) = e^{ -ks\left(\frac{\alpha^2 r(t)}{1+\alpha \beta r(t)}\right)} h_k\left(\frac{\alpha e^{-k s\left(\frac{\alpha^2 r(t)}{1+\alpha \beta r(t)}\right)+kt}}{1+\alpha \beta r(t)} x, s\left(\frac{\alpha^2 r(t)}{1+\alpha \beta r(t)}\right)  \right),
\]
and so we get that $\Tilde{h}^{\alpha,\beta}$ and $h_{k}^{\alpha,\beta}$ coincide, which implies that $h_{k}^{\alpha,\beta}$ is indeed a solution to $\mathcal{H}_k(S_{k}^{\alpha,\beta}f)$. Moreover, we have that
$h_{k}^{\alpha,\beta}(x,t) = 0  \iff x = S_{k}^{\alpha, \beta}f(t)$, because by assumption $h_{k}$ is a solution to $\mathcal{H}_k(f)$, hence $h_k(f(.),.)=0$. Also, $h_{k}^{\alpha,\beta}(x,0) =  e^{-\frac{\alpha \beta x^2}{2}}h_{k}^{\alpha,0}(x,0)$. Now, let us investigate $h_{k}^{\alpha,0}$. Using relation (\ref{OU BM solution relation}) again, we get
\[
\begin{aligned}
h_{k}^{\alpha,0}(x,t) & = \alpha e^{kt -ks(\alpha^2 r(t))} h_{k}(\alpha e^{-ks(\alpha^2 r(t))+kt}x, s(\alpha^2 r(t))) \\
& = \alpha  e^{kt -ks(e^{-2(-\ln(\alpha))} r(t))} h_{k}(e^{-ks(e^{-2(-\ln(\alpha))} r(t))+kt - (-\ln(\alpha))}x, s(e^{-2(-\ln(\alpha))} r(t))) \\
& = e^{kt} h(e^{-\epsilon} e^{kt} x, e^{-2\epsilon} r(t)) 
 = e^{kt} h^{(4)}_{\epsilon}(e^{kt}x, r(t)),
\end{aligned}
\]
where $\epsilon=-\ln(\alpha)$. Now,  the latter $h^{(4)}_{\epsilon}$ is the fourth symmetry of the heat equation listed in Section 3.3 of \cite{Alili patie ode}, meaning that it satisfies the heat equation. By using relation (\ref{OU BM solution relation}), we see that $e^{kt} h^{(4)}_{\epsilon}(e^{kt}x, r(t))$ indeed satisfies the Fokker-Plank equation of the OU and so $h_{k}^{\alpha,0}$ is a solution to $\mathcal{H}_{k}(S_{k}^{\alpha,0}f)$ and thus in particular $h_{k}^{\alpha,0}(.,0) = \delta_{0}(.)$ on $(-\infty, \frac{f(0)}{\alpha})$. Hence, we get
\[
h_{k}^{\alpha,\beta}(x,0) = e^{-\frac{\alpha \beta x^2}{2}}h_{k}^{\alpha,0}(x,0)= \delta_{0}(x),
\]
which concludes the proof. 
\end{proof}

\subsubsection{Lie symmetries of OU Fokker-Planck equation}

{In this section, we provide a direct construction of the function $h_{k}^{\alpha, \beta}$ given in (\ref{h hat alpha beta equation}) from the Lie-symmetries of the OU Fokker Planck equation (\ref{OU fokker planck equation}).} Using similar techniques discussed in Section 2.4 in \cite{Peter olver lie}, after some lengthy and tedious calculations, one finds that the Lie algebra of infinitesimal symmetries of the OU Fokker-Planck equation is spanned by six vector fields, where $x, t$ are the
two independent variables and $h$ is the dependent variable,
$$
\mathbf{v_1} = e^{-kt}  \frac{\partial}{\partial x}, \quad \mathbf{v_2} =  \frac{\partial}{\partial t}, \quad \mathbf{v_3} = h  \frac{\partial }{\partial h}, \quad
\mathbf{v_4} = e^{kt}  \frac{\partial }{\partial x} - 2kxh e^{kt} \frac{\partial }{\partial h}
$$
$$
\mathbf{v_5} = -kx e^{-2kt}\frac{\partial }{\partial x} + e^{-2kt}\frac{\partial }{\partial t}+ kh e^{-2kt} \frac{\partial }{\partial h}, \quad
\mathbf{v_6} = kx e^{2kt}\frac{\partial }{\partial x} + e^{2kt}\frac{\partial }{\partial t} - 2k^2 x^2 h e^{2kt} \frac{\partial }{\partial h}
$$
\hfill
\\
and by the infinite-dimensional sub-algebra $\mathbf{v_\alpha} = u(x,t)\frac{\partial }{\partial h}$, where $u(x,t)$ is an arbitrary solution of the OU Fokker-Planck equation. Now, by exponentiating the basis, as in page 89 of \cite{Peter olver lie}, we can produce the one-parameter group of transformations leaving invariant $\mathcal{H}_{k}(f)$. Doing this procedure for each of the vectors fields, we obtain the one-parameter groups $G_i$ generated by the $\mathbf{v_i}$. Since each group $G_i$ is a symmetry group, if $h$ is a solution to (\ref{OU fokker planck equation}), then the following are also solutions to (\ref{OU fokker planck equation}):

\begin{equation} \label{symmetry OU}
\begin{aligned}
  h^{(1)}_{k,\epsilon}(x,t) =  &   h(x- \epsilon e^{-kt}, t )\\
    h^{(2)}_{k,\epsilon}(x,t) =  &   h(x, t- \epsilon )\\
   h^{(3)}_{k,\epsilon}(x,t) =  &   e^{\epsilon} h(x, t )\\
      h^{(4)}_{k,\epsilon}(x,t) =  &   e^{-2kx \epsilon e^{kt}+k \epsilon^2 e^{2kt}}h(x- \epsilon e^{kt}, t )\\
          h^{(5)}_{k,\epsilon}(x,t) =  &   \frac{e^{kt}}{\sqrt{e^{2kt} - 2k\epsilon}} h\bigg(\frac{x e^{kt}}{\sqrt{e^{2kt} - 2k\epsilon}}, \frac{\ln{(e^{2kt}-2k\epsilon)}}{2k} \bigg)\\
          h^{(6)}_{k,\epsilon}(x,t) =  &   e^{\frac{-2k^2 x^2 \epsilon}{e^{-2kt} + 2k\epsilon}} h\bigg(\frac{x e^{-kt}}{\sqrt{e^{-2kt} + 2k\epsilon}}, \frac{\ln{(e^{-2kt}+2k\epsilon)}}{-2k} \bigg).\\
\end{aligned}
\end{equation}
{Notice that these symmetry groups provide an explanation of the invariance of the law of the OU process under some specific transformations. More precisely, $  h^{(1)}_{k,\epsilon}$ and $  h^{(2)}_{k,\epsilon}$ show the space and time invariance  of the law of the OU, respectively, $ h^{(3)}_{k,\epsilon}$ is the trivial multiplication by a constant, $ h^{(4)}_{k,\epsilon}$ represents the Girsanov transform connecting the law of the OU with different exponential drifts, while compositions of $ h^{(5)}_{k,\epsilon}$ and $ h^{(6)}_{k,\epsilon}$ are deeply connected to the change of measure connecting the law of the OU with its bridges.}

In Proposition \ref{proposition h hat alpha beta expression}, we proved that the function $h_{k}^{\alpha,\beta}$ is a solution to the boundary value problem $\mathcal{H}_{k}(S_{k}^{\alpha,\beta}f)$ by directly using the relation between the two boundary value problems in (\ref{bvps}). We can also construct this function directly using the symmetries above. 

\begin{remark} \label{remark symmetry error}

 Proposition 5 of \cite{Dmitry Lie} gives a general class of symmetries for SDEs with a drift term specified in Proposition 4 therein. The expression therein contains  misprints in $P^{\epsilon}_{1,2}$ and $\mathcal{T}_{p_{1,2}}$. Using the notation in the original article, the correct expressions are 
\begin{eqnarray*}
    P^{\epsilon}_{1,2}  & : & \frac{u(x,t)}{\mathcal{U}(\mathcal{X},\mathcal{T})} = \frac{\theta_{\mathcal{F}_2}(x)}{\theta_{\mathcal{F}_2}(\mathcal{X})}  \frac{e^{\pm 2 \sqrt{A} \left(\nu - \frac{1}{4} \pm \frac{1}{4} \right)t \mp \frac{\sqrt{A}}{2} (x+ 2B/A)^2 }}{e^{\pm 2 \sqrt{A} \left(\nu - \frac{1}{4} \pm \frac{1}{4} \right)\mathcal{T} \mp \frac{\sqrt{A}}{2} (\mathcal{X}+ 2B/A)^2}},\\
    \mathcal{T}_{p_{1,2}} & = & \mp \ln (e^{\mp 2\sqrt{A}t} + \epsilon)/(2\sqrt{A}).
\end{eqnarray*}
If we set $A= k^2, B=0, \nu=0$ and $\theta_{\mathcal{F}_2}(x)= e^{-kx^2/2}$, we recover the above symmetries (\ref{symmetry OU}) for the OU process.


\end{remark}

\begin{lemma} \label{lemma symmetry h alpha beta}
\[
    h_{k}^{\alpha,\beta}=    \left( h^{(2)}_{k,\left(\frac{-\ln(\alpha)}{k}\right)} \circ h^{(5)}_{k,\left(\frac{\alpha^2 - 1}{2k}\right)} \right) \circ \left(h^{(6)}_{k, \left(\frac{\beta}{2k(2k\alpha - \beta)}\right)} \circ h^{(2)}_{k, \left(\frac{1}{k} \ln{\left(\frac{2k\alpha - \beta}{2k \alpha}\right)}\right)} \circ h^{(5)}_{k, \left(\frac{\beta}{2k(2k\alpha-\beta)}\right)} \right).
\]
\end{lemma}
\hfill
\\
Although the calculation is tedious, it can  be easily shown that the expression given in Lemma \ref{lemma symmetry h alpha beta} holds. {Now, because the function $h_{k}^{\alpha,\beta}$ is a composition of symmetries of the OU Fokker-Planck equation, then it is itself a solution to the OU Fokker-Planck equation. Lemma \ref{lemma symmetry h alpha beta} can then be used to shorten the first part of the proof given in Proposition \ref{proposition h hat alpha beta expression}. Now, we give the third proof of Theorem \ref{theorem OU transformation}. }

\begin{proof}[Third proof of Theorem \ref{theorem OU transformation}]
{
Let $h_k$ be the solution to $\mathcal{H}_{k}(f)$. Then, by \eqref{OU BM probability solution representation}, we have $\mathbb{P}(T_{k}^f \leq t) = 1 - \int_{-\infty}^{f(t)} h_{k}(x,t) dx$. Setting $p_k^{f}(t)  dt = \mathbb{P}\Big( T_{k}^{f} \in dt \Big) $, we get
\[p_k^{f}(t) = - \frac{d}{dt} \left(\int_{-\infty}^{f(t)} h_{k}(x,t) dx \right) = -h_{k}(f(t),t) f'(t) - \int_{-\infty}^{f(t)} \frac{\partial }{\partial t} h_{k}(x,t) dx = - \int_{-\infty}^{f(t)}  \frac{\partial }{\partial t} h_{k}(x,t) dx,\]
\\
because $h_{k}$ vanishes on the boundary $f$. Using the OU Fokker-Planck equation \eqref{OU fokker planck equation}, we get
\begin{equation}\label{images density diff relation}
    p_k^{f}(t) = - \int_{-\infty}^{f(t)} - \frac{\partial }{\partial x} [-kx h_{k}(x,t)] + \frac{1}{2} \frac{\partial^2 }{\partial x^2} h_{k}(x,t) \; dx =  -\frac{1}{2}  \frac{\partial }{\partial x} h_{k}(x,t)|_{x=f(t)}.
\end{equation}
Now, by Proposition \ref{proposition h hat alpha beta expression} and Lemma \ref{lemma symmetry h alpha beta}, we deduce that $h_{k}^{\alpha,\beta}$ in (\ref{h hat alpha beta equation}) is a solution to the boundary value problem $\mathcal{H}_{k}(S_{k}^{\alpha,\beta}f)$. Hence, using $h_{k}^{\alpha,\beta}$ with its corresponding boundary $S_{k}^{\alpha,\beta}f$ in \eqref{images density diff relation} gives us expression (\ref{FPT dist relation}) from Theorem \ref{theorem OU transformation}.}

\end{proof}

\begin{remark}
{The algebraic proof based on the Lie symmetry approach yields uniqueness of the family of transformations $
S_{k}^{\alpha,\beta}$ and thus of 
the relationship 
\eqref{FPT dist relation}. Indeed, 
the function $h_{k}^{\alpha,\beta}$ is the unique solution to $\mathcal{H}_{k}(S_{k}^{\alpha, \beta} f)$, and in order to satisfy the corresponding third condition of \eqref{bvps}, we can only construct further solutions by composing it with symmetry groups (\ref{symmetry OU}) such that the time component is 0 whenever $t=0$. This only leads to a change in constants $\alpha$ or $\beta$, making $S_{k}^{\alpha,\beta}$ the only non-trivial transformation leading to \eqref{FPT dist relation}. Note that, one could use symmetry $h^{(4)}_{k, \epsilon}$ to add an extra trend to the $S_{k}^{\alpha,\beta}$ transformation, but the form of the transformation would stay the same.} 
\end{remark}
\begin{remark}
 When we take the limit as $k$ goes to 0 of the symmetries in \eqref{symmetry OU}, we only recover the first three symmetries of the heat equation, found in Section 3.3 of \cite{Alili patie ode}. The rest of the symmetries of the heat equation can be recovered by using a combination of the symmetries in \eqref{symmetry OU}.
\end{remark}

\begin{remark}
    This Lie approach led to a variant of our  boundary crossing identity (15) in \cite{Dmitry Lie} (their equation 40) up to some misprints therein. In particular, by setting $\beta' = \beta/2k \alpha, \alpha' = 1/\alpha^2 - \beta/2k \alpha -1, A=k^2, B=0, \nu=0, x_0=0$ and $\theta_{\mathcal{F}_2}(x)= e^{-kx^2/2}$, we noted that the first term in (40) should be $\mathcal{T}_-^{-1}\mathcal{T}_+^{-1/2}\sqrt{\alpha'+\beta'+1}$ instead of $\mathcal{T}_+^{-1/2}/\sqrt{\alpha'+\beta'+1}$.
\end{remark}

\section{Asymptotic behaviour of the FPT densities} \label{section asymptotic behaviour of FPT densities}

Here, we discuss the asymptotic behaviour of the FPT distribution of $T_{k}^{S_{k}^{\alpha,\beta}f}$ {and the transience  of the transformed curve $S_{k}^{\alpha,\beta}f$, i.e.,  $\mathbb{P}(T_{k}^{S_{k}^{\alpha,\beta}f} < \infty)<1$}. As before, we assume that our boundary $f \in \mathcal{C}^1([0, \infty), \mathbb{R}^+), f(0)\neq 0$ and only consider $\alpha>0$, as discussed before Section \ref{Section5.1}. Write $\mathbb{P}(T_{k}^f \in dt) = p_k^f(t)dt$. {If $\zeta_{k, \alpha, \beta} = \infty$}, from Theorem \ref{theorem OU transformation}, we get the following asymptotic identity
\begin{equation} \label{general asymp formula} p_k^{S_{k}^{\alpha,\beta}f}(t)\sim  e^{2kt} \alpha^2 \frac{(\alpha \beta r(t)+1)^{-1/2}}{2k \alpha^2 r(t)+ \alpha \beta r(t)+1}  e^{- \frac{\alpha \beta }{2(\alpha \beta r(t)+1)} (S_{k}^{\alpha,\beta}f(t))^2 \; e^{2kt}}  p_k^{f} \left(s\left(\Gamma_{\alpha,\beta}\right)\right)  \; \; \text{as} \; t \to \infty,
\end{equation}
{
where 
$$
\Gamma_{\alpha,\beta} := \begin{cases}
\frac{\alpha}{\beta} & \text{if} \quad \beta >0, \; k \geq 0;\\   
\frac{\alpha^2}{\alpha \beta - 2 k} &\text{if} \quad \beta> \frac{2k}{\alpha} - 2k \alpha,\; k<0,\\   
\end{cases}
$$
}
and $h(t) \sim I(t)$ as $t \to \zeta$ denotes that $h(t)/I(t) \to 1$ as $t \to \zeta$, for some $\zeta \in [0,\infty]$. 
\newline
We now give a necessary and sufficient condition for a curve to be transient {for the OU process.}

{\begin{theorem}
   $S_{k}^{\alpha, \beta}f$ is transient in the following two cases
\begin{enumerate}[label=(\roman*)]
    \item  $\zeta_{k, \alpha,\beta}<\infty$;
    \item  $\zeta_{k, \alpha,\beta}=+ \infty$ and $0< \left(\beta - \frac{2k}{\alpha} \mathbbm{1}_{\{k<0\}} \right) \sqrt{2k\Gamma_{\alpha,\beta}+1} f(s(\Gamma_{\alpha,\beta})) < \infty$.
\end{enumerate}
    \end{theorem}}

{\begin{proof}
 In case $(i)$, there is always a positive probability of the OU process never hitting the curves, making them transient. We now proceed with proving the result for case (ii). By relation (\ref{OU brownian stopping time relation}), we deduce that
\[
\mathbb{P}(T_{k}^f < \infty) = \mathbb{P}(s(T^{\Lambda_{k} f})< \infty) =  \mathbb{P}(T^{\Lambda_{k} f}< \infty).
\]
Now, by the classical Kolmogorov-Eröds-Petrovski theorem reported in \cite{Erdos}, if $t^{-1/2} \Lambda_{k} f(t)$  (i.e.  $f(t) e^{kt}/\sqrt{r(t)}$) is increasing for sufficiently large $t$, then $\Lambda_{k} f$ is a transient curve for the Brownian motion if and only if
\begin{equation} \label{KPE OU integral test}
    \int_{1}^{\infty} t^{-3/2} \Lambda_{k} f(t) e^{- (\Lambda_{k} f)^2/2t} dt =  \int_{s(1)}^{\infty} \frac{e^{3kt}f(t)}{r(t)^{3/2}} e^{- e^{2kt} f(t)^2/2r(t)} dt < \infty.
\end{equation}
        Since $\zeta_{k, \alpha, \beta}= + \infty$, note that
\begin{equation} \label{asymptotic g function}
   g(t) : = e^{kt} S_{k}^{\alpha,\beta}f(t)/r(t) \sim \left(\beta - \frac{2k}{\alpha} \mathbbm{1}_{\{k<0\}} \right) \sqrt{2k(\Gamma_{\alpha,\beta})+1} f(s(\Gamma_{\alpha,\beta})) > 0 \quad \text{as}  \; \; t \to \infty.
\end{equation}
 Furthermore, because of (\ref{asymptotic g function}),  $S_{k}^{\alpha,\beta}f(t) e^{kt}/\sqrt{r(t)}$ is increasing for sufficiently large $t$, and under the condition $0< \left(\beta - \frac{2k}{\alpha} \mathbbm{1}_{\{k<0\}} \right) \sqrt{2k\Gamma_{\alpha,\beta}+1} f(s(\Gamma_{\alpha,\beta})) < \infty$, we have
\begin{eqnarray*}
    \int_{s(1)}^{\infty} \frac{e^{3kt}S_{k}^{\alpha,\beta}f(t)}{r(t)^{3/2}} e^{- e^{2kt} (S_{k}^{\alpha,\beta}f(t))^2/2r(t)} dt &=& \int_{1}^{\infty} \frac{1}{\sqrt{z}} g(s(z)) e^{- g(s(z))^2 z/2} dz \\
    &<&  \int_{1}^{\infty}  g(s(z)) e^{- g(s(z))^2 z/2} dz < \infty.
\end{eqnarray*}
The last integral is finite because $g$ eventually stabilises and the assumptions ensure that $g$ never blows up or vanishes, so its minimum and the maximum are always attained. Hence by the integral test (\ref{KPE OU integral test}), $S_{k}^{\alpha,\beta}f$ is transient.
    \end{proof}}
 {Now, we derive asymptotics for when $\zeta_{k, \alpha, \beta} <  \infty$. Notice that in this case,} $s\left( \frac{ \alpha^2 r(t)}{1 + \alpha \beta r(t)} \right) \to +\infty$ as $t \to \zeta_{k, \alpha, \beta}$ and whenever $f$ is transient and satisfies suitable conditions, we can derive the asymptotic expression of the FPT density of the OU {hitting curves $S_{k}^{\alpha, \beta}f$} by making use of the corresponding result for the BM by Anderson and Pitt [4], as proved in the following Proposition.

\begin{proposition} \label{anderson pitt OU}
    Assume that $f$ is transient and satisfies the following conditions
    \begin{enumerate}[label=(\roman*)]
        \item $\Lambda_{k} f$ is increasing, concave, twice differentiable on $(0,\infty)$ and of regular variation at $\infty$ with index ${a} \in [1/2,1)$,
        \item $\Lambda_{k} f(t)/\sqrt{t}$ is increasing at $\infty$, and $\Lambda_{k} f(t)/t$ is convex and decreases to 0 for sufficiently large t,
        \item There exist positive constants $c<1$ and $c'$ such that the inequalities \\ $t f'(s(t)) \leq (c(2kt+1)-kt) f(s(t))$ and $| t^2/(2kt+1)^2 (f''(s(t))-k^2 f(s(t))) | \leq c' f(s(t))$ are met for sufficiently large t.
      \end{enumerate}
      Then, {if $\zeta_{k, \alpha, \beta} < \infty$}, we have
    \[
    p_k^{S_{k}^{\alpha,\beta}f} (t) \sim \sqrt{\frac{|\alpha \beta|^3}{2 \pi}} \left(\frac{2k}{|\alpha \beta|} + 1\right) (1-r) \Tilde{f}(t) \quad \text{as} \; \; t \to \zeta_{k, \alpha,\beta},
    \]
where
\[
\Tilde{f}(t) = \frac{ \left(\frac{ k\alpha^2 r(t)}{1 + \alpha \beta r(t)}+1\right)f\left(s\left( \frac{ \alpha^2 r(t)}{1 + \alpha \beta r(t)} \right)\right) - \left( \frac{ \alpha^2 r(t)}{1 + \alpha \beta r(t)}\right)f'\left(s\left( \frac{ \alpha^2 r(t)}{1 + \alpha \beta r(t)} \right)\right)}{\sqrt{2k  \frac{ \alpha^2 r(t)}{1 + \alpha \beta r(t)}+1}}.
\]

\end{proposition}

\begin{proof}

{
By Theorem 1 of \cite{Anderson pitt}, if conditions (i)-(iii) are satisfied, we can write
\[
    p^{\Lambda_{k} f}(t)\sim  (1-r) \frac{\Lambda_{k} f(t) - t (\Lambda_{k} f(t))'}{\sqrt{2 \pi} t^{3/2}} e^{-(\Lambda_{k} f(t))^2/2t} \quad \text{as} \; \; t \to \infty.
\]
Using (\ref{OU brownian stopping time relation}), we get $p_k^f(t) = e^{2kt} p^{\Lambda_{k} f} (r(t))$. 
We then simplify the expression and combine it with (\ref{FPT dist relation}). The result follows by noticing that the assumption $\Lambda_{k} f(t)/t \downarrow 0$ as $t \to \infty$ forces the exponential terms to vanish. 
}
\end{proof}


\begin{remark}
We noted a misprint in Section 4.3 of \cite{Alili patie time inversion}. For the Brownian motion, when $\beta>0$, the one parameter family of transformed functions $S^{1, \beta}f$ is always transient whenever $t^{-1/2} S^{1, \beta}f$ is increasing for sufficiently large t, $0< \beta f(\alpha/\beta)< \infty$ and $f(0)>0$.
\end{remark}

\section{Interpretation via the method of images} \label{section method of images}

In \cite{Lerche}, the author did a thorough investigation of the method of images for the standard BM. Here, we apply the method of images to the OU process and use it to produce new examples of curves with explicit FPT densities. As in the BM case, we would like to construct a function $h_k$ satisfying the OU Fokker-Planck equation. Then, by the uniqueness of such solutions, as seen in Section \ref{section lie proof}, $h_k$ would also satisfy (\ref{OU BM probability solution representation}) and be the solution to $\mathcal{H}_{k}(f)$. We proceed as follows. First, we assume to have a positive $\sigma-$finite measure $F$ with $\int_{0}^{\infty} \phi(\sqrt{\epsilon \theta})F(d \theta)<\infty$ for all $\epsilon>0$. Then, for $a>0$, we define the $h_k$ function by
\[
 \begin{aligned}
       h_{k}(x,t) & := p_{t}(0,x) - \frac{1}{a} \int_{0}^{\infty} p_{t}(\theta,x) F(d \theta) 
       =  \frac{e^{kt}}{\sqrt{r(t)}} \phi\left( \frac{x e^{kt}
    }{ \sqrt{r(t)}} \right) - \frac{1}{a} \int_{0}^{\infty} \frac{e^{kt}}{\sqrt{r(t)}}\phi\bigg(\frac{x e^{kt} - \theta }{\sqrt{r(t)}}\bigg) F(d \theta).
 \end{aligned}
\]
We know that $ h_{k}$ vanishes on the boundary $f$, and so by simplifying, we get that $ h_{k}(x,t)=0$ is equivalent to
\begin{equation} \label{h equation simplified method of images}
   l\left(\frac{x e^{kt}}{r(t)}, \frac{1}{r(t)}\right) = a,
\end{equation}
where $l(y,s) = \int_{0}^{\infty} e^{\theta y - \frac{1}{2} \theta^2 s} F(d \theta)$. 

The following lemma gives us a characterisation of the boundaries obtained through the method of images.
\begin{lemma}[Characterisation of boundaries] \label{lemma characterisation}
The boundaries $f$ obtained through the method of images have the following properties:

\begin{enumerate}
    \item $f$ is infinitely often continuously differentiable;
    \item $\Lambda_{k} f(t)/t$ (i.e. $f(t) e^{kt}/r(t)$) is monotone decreasing;
    \item $f''(t) - k^2 f(t) \leq 0$.
\end{enumerate}

\end{lemma}

\begin{proof}
In Lemma 1.1 of \cite{Lerche}, it is shown that the curves $\eta(t)$ that satisfy the equation $l(\frac{\eta(t)}{t}, \frac{1}{t}) = a$ are infinitely continuously differentiable and concave, with $\eta(t)/t$  monotone decreasing. As in our case we have $a = l\left(\frac{f(t)e^{kt}}{r(t)} , \frac{1}{r(t)}\right)$, by changing $t \to s(t)$, we get $a = l\left(\frac{\Lambda_{k} f(t)}{t} , \frac{1}{t}\right)$. So  $\eta(t) := \Lambda_{k} f(t)$ must satisfy the three properties given in Lemma 1.1 of \cite{Lerche}, so (1) and (2) follow directly. For (3),
\[
f''(t) = k^2 e^{-kt}  \eta(r(t)) + e^{3kt} \eta''(r(t)) = k^2 f(t)  +  e^{3kt} \eta''(r(t))
\]
and the result follows as $\eta(t)$ is concave.
\end{proof}

\begin{remark}
From {Lemma} \ref{lemma nonlinear ODE}, if $\mu(s(\tau f)) > k^2$ and $f\geq 0$, then we get that the solutions to nonlinear ordinary differential equations (\ref{nonlinear ODE}) satisfy $f'' - k^2 f \leq 0$.
\end{remark}

The following theorem gives us a way of calculating the density of $T_{k}^f$ explicitly.

\begin{theorem} \label{theorem FPT density method of images}
The FPT density of $T_{k}^f$ is given by
\[
    \mathbb{P}(T_{k}^f \in dt)= \frac{ e^{2kt}}{2 r(t)^{3/2}} \phi\bigg(\frac{f(t)e^{kt} }{\sqrt{r(t)}} \bigg) E(\theta | (f(t),t))dt,
\]
where
\[
 E(\theta | (f(t),t)) = \frac{\int_{0}^{\infty} \theta \phi\bigg(\frac{f(t)e^{kt} - \theta}{\sqrt{r(t)}}\bigg) F(d\theta)}{\int_{0}^{\infty} \phi\bigg(\frac{f(t)e^{kt} - \theta}{\sqrt{r(t)}}\bigg) F(d \theta)}.
\]
\end{theorem}

\begin{proof}
 We just use relation (\ref{images density diff relation}) to get the result.
\end{proof}
We now consider some examples of the boundaries that arise for specific measures $F$ via Theorem \ref{theorem FPT density method of images}.
\begin{example} \label{example images OU 1}
Consider $F(d\theta) = \delta_{2z}$ for some $z>0$. Using $h_k(f(t),t)=0$ and simplifying, we get
\[
 f(t) = \frac{\ln{(a)}}{2 z k} \sinh{(kt)} + z e^{-kt}.
\]
By Theorem \ref{theorem FPT density method of images},
   \[
            \mathbb{P}(T_{k}^f \in dt)= \frac{ z e^{2kt}}{ r(t)^{3/2}} \phi\left(\frac{\frac{\ln(a)}{2z}r(t) +z }{\sqrt{r(t)}} \right) dt.
        \]
\end{example}
\begin{remark}
In \cite{Buonocore OU hyperbolic boundary}, the FPT of a mean-reverting OU with parameter $\mu$ hitting a hyperbolic type boundary of the form $\mu + A e^{kt} + B e^{-kt}$ is studied. This reduces to the FPT of a standard OU hitting a curve of the form $A e^{kt} + B e^{-kt}$ for arbitrary constants A and B which is the curve given in this example. Also, applying the $S_{k}$ transformation to these curves gives the same family of curves with different constant coefficients.
\end{remark}
If the support of the measure $F$ is on $\mathbb{R}$ and $F(\{0\})=0$, then we define 
\[
  h_{k}(x,t) := p_{t}(0,x) - \frac{1}{a} \int_{-\infty}^{\infty} p_{t}(\theta,x) F(d \theta).
\]
Then, there exist positive and negative valued functions $f_{+}$ and $f_{-}$, with $f_{-} < f_{+}$ and the properties $ h_{k}(f_{+}(t),t)=0$ and $ h_{k}(f_{-}(t),t)=0$ for all $t < t_a$ for a certain $t_a \leq \infty$. Then, the stopping time can also be defined as the first exit time from the region $(f_{-}(t), f_{+}(t)) $, i.e.
\begin{equation}\label{Ftwobounaries}
T_{k}^{f_{\pm}}= \inf \{0<t<t_a ; \;  \; U_t \notin (f_{-}(t), f_{+}(t)) \}.
\end{equation}
{Moreover, if $F$ is symmetric, then $f_+=-f_-$, and the FPT distribution of $T_{k}^{S_{k}^{\alpha,\beta}f_\pm}$ can be obtained using our main identity \eqref{FPT dist relation} 
in Theorem \ref{theorem OU transformation}, replacing  $T_{k}^f$ and $T_{k}^{S_{k}^{\alpha,\beta}f}$  with  $T_{k}^{f_\pm}$ and $T_{k}^{S_{k}^{\alpha,\beta}f_\pm}$, respectively. This follows directly by replacing the processes $B$ and $U$ with $|B|$ and $|U|$, respectively. }
We now see an example of this.
\begin{example} \label{example OU 3}
Consider  $F(d\theta) = \frac{ d \theta}{\sqrt{2 \pi}}$ on $\mathbb{R}$. Then,
\[
 f_{\pm}(t)= \pm e^{-kt} \sqrt{r(t) \ln\left(\frac{a^2}{r(t)}\right)}, \quad 0<t \leq t_a =  
 \begin{cases}
s(a^2) & \text{if } k\geq0 \text{ or } k<0, a^2<-1/(2k), \\      
  + \infty & k<0, a^2\geq -1/(2k). 
\end{cases}
\]
This is a two sided boundary, with exit time $T_{k}^{f_{\pm}}$ defined in \eqref{Ftwobounaries} and FPT density given by
\[
 \mathbb{P}(T_{k}^{f_\pm} \in dt)=  \frac{ e^{2kt} }{2 r(t)} \phi\left( \sqrt{\ln\left(\frac{a^2 }{r(t)}\right)} \right)  \sqrt{\ln\left(\frac{a^2 }{r(t)}\right)} dt.
\]
Applying the $S_{k}$ transformation (\ref{S hat transformation}) to this curve, we get
\begin{equation} \label{images last example transformation}
      S_{k}^{\alpha,\beta} f_{\pm}(t)  = 
     \pm \sqrt{r(t)} \sqrt{1+ \alpha \beta r(t)} e^{-kt} \sqrt{\ln{\bigg[\frac{a^2 (1 + \alpha \beta r(t)) }{ \alpha^2 r(t)} \bigg]}}, \quad t< \zeta_{\alpha,\beta,a},
\end{equation}

where

$$
 \zeta_{\alpha, \beta, a} = \begin{cases}
s\left(- \frac{a^2}{a^2 \alpha \beta - \alpha^2}\right) & \text{if} \quad a^2 \beta < \alpha;\\      
  + \infty & \text{otherwise}.
\end{cases}
$$

The FPT density is given by
\[
\begin{aligned}
 p_k^{S_{k}^{\alpha,\beta} f_{\pm}}(t) 
 & = e^{2kt} \alpha^2 \frac{(1+ \alpha \beta r(t))^{-3/2}}{\big(2k \frac{\alpha^2 r(t)}{1+ \alpha \beta r(t)} + 1 \big)}  e^{- \frac{\alpha \beta }{2(1 + \alpha \beta r(t))} ( S_{k}^{\alpha,\beta} f_{\pm}(t))^2 \; e^{2kt}}  \; 
 p_k^{f_{\pm}}\left(s\left( \frac{ \alpha^2 r(t)}{1 + \alpha \beta r(t)} \right) \right) \\
 & =   \frac{e^{2kt}}{\sqrt{8 \pi}} \sqrt{\frac{\ln{\bigg[\frac{a^2 (1 + \alpha \beta r(t)) }{ \alpha^2 r(t)} \bigg]}}{1+ \alpha \beta r(t)}}  e^{ - \frac{1}{2} \ln{\left[\frac{a^2 (1+\alpha \beta r(t)) }{ \alpha^2 r(t)} \right]} (\alpha \beta r(t) +1)} .
\end{aligned}
\]
Now, for example, by setting $a=1, \alpha=2, \beta=k$, we get
\[
   S_{k}^{2,k} f_{\pm}(t)  = \pm \sqrt{r(t)} \;  \sqrt{\ln{\bigg[\frac{k }{ 2 (1 - e^{-2kt})} \bigg]}},
\]
with FPT density given by
\[
  p_k^{ S_{k}^{2,k} f_{\pm}}(t) = \frac{e^{kt}}{\sqrt{8 \pi}} \sqrt{\ln{\bigg[\frac{k }{ 2 (1 - e^{-2kt})} \bigg]}} e^{ - \frac{e^{2kt}}{2}\ln{\left[\frac{k }{ 2 (1 - e^{-2kt})} \right]}}.
\]
{In Figure \ref{figure transformed closed}, we report the transformed curves (\ref{images last example transformation}) for different values of $\alpha$ and $\beta$. In particular, for $a^2\beta < \alpha$, we get closed shaped curves defined on $t \in (0,\zeta_{\alpha, \beta, a})$, approaching $0$ as $t\to 0$ or $t\to \zeta_{\alpha,\beta,0}$, see the left panel.  Otherwise, we get open curves that diverge to $\pm \infty$, see the right panel.} 
\begin{figure}[t]
    \centering
 \includegraphics[width=.8\textwidth]{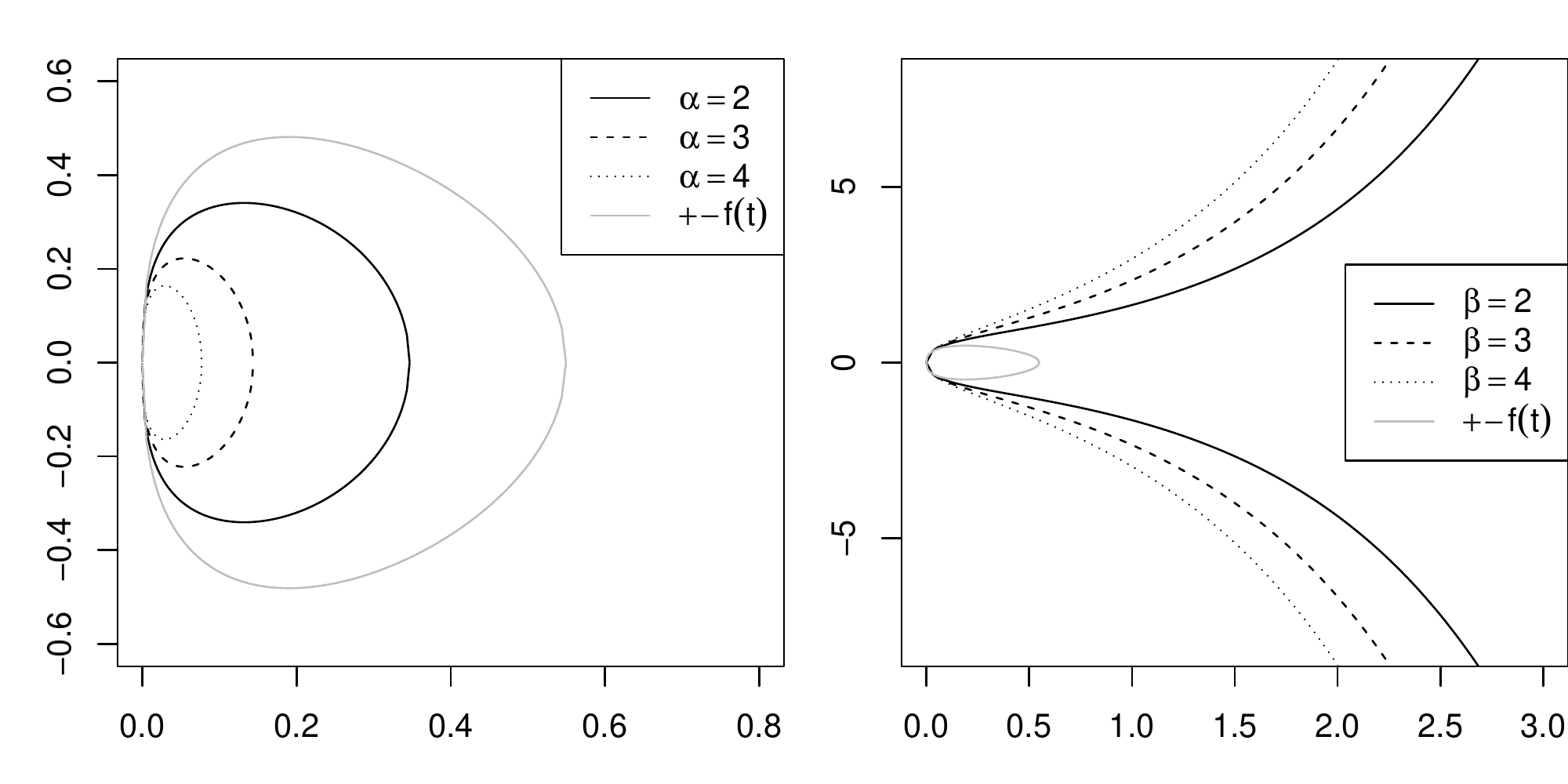}
    \caption{Transformed curves (\ref{images last example transformation}) {when $a^2\beta < \alpha$ (left panel) and $a^2\beta >  \alpha$ (right panel). In particular, we choose:  $\beta=1, k=1, a=1$ with $\alpha= 2,3,4$ in the left panel, and $k=1, a=1, \alpha = 1$ and $\beta=2,3,4$ in the right panel.}}
    \label{figure transformed closed}
\end{figure}
\end{example}

\begin{remark}

Back to our $S_{k}^{\alpha,\beta}$ transformation, using (\ref{h equation simplified method of images}) with the time-transformation $t \rightarrow s\bigg( \frac{ \alpha^2 r(t)}{1 + \alpha \beta r(t)} \bigg)$, we get
\[
   a =  \int_{0}^{\infty} e^{\frac{\theta}{\alpha} [\frac{S_{k}^{\alpha,\beta} f(t) e^{kt}}{r(t)}]  - \frac{1}{2} [\frac{(\theta/\alpha)^2}{r(t)}]}  F^{\alpha, \beta}(d\theta),
\]
where $F^{\alpha, \beta}(d\theta) =  e^{-\frac{\beta}{\alpha} \frac{\theta^2}{2} } F(d \theta) $ is the measure corresponding to the curve  $S_{k}^{\alpha,\beta} f(t)$, {for $t<\zeta_{k, \alpha,\beta}$}.

\end{remark}

\section*{Acknowledgements}
We are grateful to Dmitry Muravey for a discussion on the Lie symmetries while finalising the paper.
AA was funded by EPSRC grant [EP/V520226/1] as
part of the Warwick CDT in Mathematics and Statistics.  For the purpose of open access, the authors have applied a Creative Commons Attribution (CC BY) licence to any Author Accepted Manuscript version arising from this submission.

\end{document}